\documentclass[11pt]{article}
\usepackage{amsthm}
\usepackage{amssymb}
\usepackage{epsfig}
\usepackage{amsthm}
\usepackage{srcltx}
\usepackage[dvipsnames,usenames]{color}
\usepackage{paralist}
\usepackage{amsmath}
\usepackage{amsfonts}
\usepackage{float}
\usepackage{soul}
\usepackage{hyperref}
\usepackage{mathabx} 

\newcounter{contador}
\newcounter{teoA}

\newtheorem{teoa}[teoA]{Theorem}

\newtheorem{propo}[contador]{Proposition}
\newtheorem{teo}[contador]{Theorem}
\newtheorem{lem}[contador]{Lemma}
\newtheorem{defi}[contador]{Definition}
\newtheorem{corol}[contador]{Corollary}

\theoremstyle{remark}
\newtheorem{nota}[contador]{Remark}

\newcounter{ex}

\newcommand{\R}{{\mathbb R}}

\newcommand{\N}{{\mathbb N}}

\newcommand{\Z}{{\mathbb Z}}
\title{Pointwise periodic maps with quantized first integrals}

\author{Anna Cima$^{(1)}$, Armengol Gasull$^{(1,2)}$, V\'{\i}ctor Ma\~{n}osa$^{(3)}$ and Francesc Ma\~{n}osas$^{(1)}$
    \\*[.1truecm]
    {\small \textsl{$^{(1)}$ Departament de Matem\`{a}tiques, Facultat
            de Ci\`{e}ncies,}}
    \\*[-.25truecm] {\small \textsl{Universitat Aut\`{o}noma de Barcelona,}}
    \\*[-.25truecm] {\small \textsl{08193 Bellaterra, Barcelona,
    Spain}}\\
 \\*[-.25truecm] {\small \textsl{$^{(2)}$ Centre de Recerca Matem\`{a}tica, Campus de
Bellaterra,}}
    \\*[-.25truecm] {\small \textsl{08193 Bellaterra, Barcelona, Spain}}
    \\*[-.25truecm] {\small \textsl{cima@mat.uab.cat,
            gasull@mat.uab.cat, manyosas@mat.uab.cat}}\\
    \\*[-.25truecm] {\small \textsl{$^{(3)}$ Departament de Matem\`{a}tiques,}}
     \\*[-.25truecm] {\small \textsl{Institut de Matem\`{a}tiques de la UPC-BarcelonaTech (IMTech),}}           
    \\*[-.25truecm] {\small \textsl{Universitat Polit\`{e}cnica de Catalunya}}
    \\*[-.25truecm] {\small \textsl{Colom 11, 08222 Terrassa, Spain}}
    \\*[-.25truecm] {\small \textsl{victor.manosa@upc.edu}}}

\date{}

\begin{document}

% ********************** EN CAS D'ARTICLE *********************
\maketitle
\begin{abstract}   We describe the global dynamics of some pointwise periodic piecewise linear maps in the plane that exhibit interesting dynamic features. For each of these maps we find a first integral. For these integrals the set of values are
discrete, thus quantized. Furthermore, the level sets are bounded sets whose interior
is formed by a finite number of open tiles of certain
regular or uniform tessellations. 
The action of the maps on each invariant set of tiles is described geometrically.

\end{abstract}

%% *************************************************************
%
%

\noindent {\sl  Mathematics Subject Classification 2010:}
Primary: 37C25, 39A23. Secondary: 37C55, 37J35, 52C20.

\noindent {\sl Keywords:}  Periodic points; pointwise periodic maps;
piecewise linear maps; quantized first integrals; regular and uniform
tessellations.

\newpage

\section{Introduction}

%$\triangle, \square,   \hexagon, \pentagon$

A \emph{pointwise periodic map} is a bijective self-map in a
topological space  such that each point is periodic. A
\emph{periodic map} is a bijective self-map in a topological space
such that some iterated of the map is the identity.  For a periodic
map $F:X\longrightarrow X$ the minimum natural number $p$ satisfying
$F^p=\mathrm{Id}$ is called \emph{the period} of $F.$ Notice that a
pointwise periodic map satisfying that the period of the points has
an upper bound is periodic and its period is the least common
multiple of the periods of the elements of the space.

A classical result of Montgomery establishes that any
\emph{pointwise periodic} \emph{homeomeorphism} in an Euclidean
space is \emph{periodic}, \cite{Mont}. Non-periodic but pointwise periodic bijective maps do exist when the continuity assumption is relaxed, see \cite{VS} for instance. In the series of papers
\cite{ChaChe13,ChaChe14} and \cite{ChaWanChe12}, the authors
introduce three explicit examples of pointwise
periodic maps that are not periodic. The examples given by
these authors in the above mentioned references belong  to the
family of piecewise affine
maps with a line of discontinuity:
\begin{equation}\label{E:Grho}
G(x,y)=\left(y,-x-\rho y+\operatorname{sign}(y)\right),\quad
\mbox{where}\quad\operatorname{sign(y)}=\left\{
 \begin{array}{ll}
 +1, & \hbox{if $y\ge0$;} \\
 -1, & \hbox{otherwise,}
  \end{array}
  \right.
\end{equation} for $|\rho|<2.$ In particular they
correspond to the cases $\rho\in\{-1,0,1\}$. There are other values
of~$\rho$ for which there exist non-periodic points, see
\cite{ChaCheYeh2019}. Notice also that maps \eqref{E:Grho}
correspond to the second order discontinuous difference equations
$x_{n+2}=-x_n-\rho x_{n+1}+\operatorname{sign}(x_{n+1}).$

As we will see in next section, each map $G$ is linearly conjugate
with the piecewise rotation map
\begin{equation}\label{e:normalform}
F(x,y)= \left(\begin{array}{rr}
\cos(\alpha)&\sin(\alpha)\\
-\sin(\alpha)&\cos(\alpha)
\end{array}\right)\left(\begin{array}{ll}
x-\operatorname{sign}(y)\\
y
\end{array}\right),
\end{equation}
where $\rho=-2\cos(\alpha)$ with $\alpha\in(0,\pi).$ Observe that
the maps $G$ with $\rho=-1,0$ and $1$ are conjugate with the maps
$F$ with $\alpha={\pi}/{3},{\pi}/{2}$ and ${2\pi}/{3},$
respectively. As we will see, the normal form $F$ regularizes the
shape of the invariant sets and keeps the same discontinuity line
$y=0.$  These maps are included in the class of symmetric maps studied in the remarkable paper \cite{GQ} 
together with other more general piecewise rotations,  see a further comment below. As noticed in \cite{BG}, they exhibit complex dynamics and they belong to the type of piecewise rotations with the same rotation angle that elude the
generic dichotomy that  appears  in most piecewise rotations   of being globally attracting or globally repelling maps, see \cite[Theorem 1]{BG}.

Piecewise affine
maps with a line of discontinuity appear as
models in many fields like in the study of mechanical systems with
friction, power electronics, relay control systems or economics
\cite{BV01,B99,ZM03}. In fact, as is explained in \cite{ChaChe13,ChaChe14,ChaWanChe12}, the three maps \eqref{E:Grho} with $\rho\in\{-1,0,1\}$
appear in the study of steady states of certain
cellular neural networks. Despite their apparent simplicity,
piecewise affine maps exhibit great dynamic richness and a variety
of phenomena that are characteristic of these systems, see
\cite{BBCK,BG,GQ,SW08,ZM03}  and references therein.  As we will show, the examples 
considered in this paper are also very rich from a dynamical viewpoint,
even though each orbit is periodic. In fact, one of our motivations
was to highlight the beautiful features of these
examples.

Recall
that a first integral of a discrete dynamical system associated with
a map $F$ is a non-constant real valuated function $V$ such that
$V\circ F=V$, which means that the level sets $\{V=c\}$, typically
called the energy levels, are invariant under the action of
the map. 
It is known that periodicity issues are related
with  integrability since most continuous periodic maps are
completely integrable (there exist as many functionally independent first integrals as the dimension of the phase space), see \cite{CGM06} and \cite{CGMM16}.  In this
work we consider the piecewise affine maps $F$ with $\alpha\in\{\pi/3,\pi/2,2\pi/3\}$ under the light of their properties as
integrable systems.  For each of these three maps, we obtain a non-trivial
\emph{first integral} which is defined in an open and dense set
of $\R^2$ and have discrete (or quantized) energy levels. 
Then we describe their global features in terms of the
dynamics induced by the maps on the level sets of the first
integrals.  These level sets are bounded, with positive measure and their
interior is formed by a finite number of some prescribed tiles of
certain regular or uniform tessellations  forming necklaces, see Figures \ref{f:T90}, \ref{f:T120}
and \ref{f:T60}. The existence of necklaces in piecewise isometries is well known. For instance, in \cite[Theorem 1 and Lemma 11]{GQ}
it is established the existence of some invariant necklaces defined by convex polygons containing periodic islands for a family of maps that contain the ones studied in this papers. These necklaces and the set of periods associated with their periodic orbits are characterized both analytically and geometrically, and its existence is the key to prove the boundedness of the orbits of the maps considered there. We want to point out that in our maps, all the integral's level sets are necklaces. In Remark \ref{r:final} we comment the relation between both families of necklaces. In addition, as we will see, the maps  $F$ with $\alpha\in\{\pi/3,\pi/2,2\pi/3\}$ have also a second continuous  first integral, see Remark \ref{r:dist}. This second first integral, however, is not useful to control the set of periods.

Planar piecewise isometries appear in the study of polygonal dual billiards (\cite{DT,GuSi,Moser}). The results in the literature
indicate that  some polygonal dual billiards should  also have quantized integrals,
see Figure 3 in \cite{Moser},  Figure 2 in \cite{DT}, or Figures 3--5 and the results in \cite{VS}. We believe that the explicitness  of the analytic expression of the quantized integrals
with positive measure level sets for the maps \eqref{e:normalform} is quite novel in the
context of discrete dynamical systems theory.  
It is interesting to notice the fact that the regular tessellations that we find in this paper also appear in the study of some polygonal dual billiards like the one introduced by Moser in \cite{Moser} or those that appear in \cite{VS}. Observe, however, that these dual billiards are
not conjugated to any map considered in our paper, because they exhibit different sets of periods.

A consequence of our results for $F,$ when
$\alpha\in\{\pi/3,\pi/2,2\pi/3\},$ is the existence of an open and
dense subset $\mathcal{U}$ on which the dynamics of the map is
strongly stable and simple. We will see that for any $x\in
\mathcal{U}$ there exists an open neighborhood of $x$, say $\mathcal{U}_x$,
and $n_x\in \N$ such that
$F^{n_x}\vert_{\mathcal{U}_x}=\operatorname{Id}.$ Moreover, varying
$x\in\mathcal{U},$ the values $n_x$ are unbounded.

We will study the three cases separately in three different
sections. In a few words, the main results that we will state in detail in the
next section, are:
\begin{enumerate}[(a)]
\item We present  first integrals $V$ for each case.
See Section \ref{s:firstintegrals} for a constructive approach for
obtaining them.

\item The interior of the level sets of each first integral is described in terms of some prescribed open \emph{tiles} of a
regular or uniform tessellation of $\R^2$, see Figures \ref{f:T90},
\ref{f:T120} and \ref{f:T60}. In all cases, each of them is a
necklace whose beads  (the open tiles) are open sets having one of the following three
shapes: squares ($\alpha=\pi/2),$ triangles and hexagons
($\alpha\in\{\pi/3,2\pi/3\}$). In the three figures the beads of a
necklace have the same color. In fact, the shape and the number of
beads, say $M,$ only depend on the level set $k$ and $\alpha.$
Moreover, the inter-tile dynamics can be described in a very simple way:
if we collapse each of the open tiles in a point, the
interior of $\{(x,y): V(x,y)=c\}$ can be identified with $\Z_M,$
simply following the order given by the necklace in clockwise sense, were, as usual, given $q\in\Z,$ we
 denote by $\Z_q$ the set of the residue classes induced by the congruence $n\equiv m$
 if and only if $n-m$ is a multiple of $q$ with $n,m\in\Z$.
Then we will prove that the dynamical system generated by $F,$
restricted to this set, is conjugated to an affine discrete
dynamical system generated by a map $h:\Z_M\to\Z_M$, where $h(i)= i
+u(c,\alpha),$ for some $u(c,\alpha)\in\Z_M$ that we also determine explicitly
in this paper, see Theorems~\ref{th:a},~\ref{th:b} and~\ref{th:c}.
Notice that, geometrically, $F$ acts as a rotation among the beads
of any necklace. A similar inter-tile dynamics' description in the context of dual polygonal billiards can be found in \cite{GuSi}, and  also in the context of piecewise linear maps \cite[Theorem 1 and Lemma 11]{GQ}.

Due to the above conjugation,  the dynamics on the interior of each
level set can be completely understood, see Theorems \ref{th:a},
\ref{th:b} and \ref{th:c}. Roughly speaking, for each map and for
each necklace  (set of tiles with the same energy level), there
exists a  certain number $k\in\{M,M/2\}\cap\N$, that depends
(explicitly) on the energy level, so that each tile is invariant by
$F^k$. Furthermore, \emph{on each tile, $F^k$ is a rotation of
order~$p$ around the center of the tile,} where $p\in\{2,4\}$ when
$\alpha=\pi/2$, or $p\in\{3,6\}$  when $\alpha\in\{\pi/3,2\pi/3\}$
and it is determined explicitly by the energy level. As a
consequence, on each tile there is a $k$-periodic point (the center)
and the rest of points are $kp$-periodic. The dynamics in the necklaces is, therefore, a discrete version of an epicyclic motion around a discrete deferent which is the locus of the centers of the tiles, \cite[p.~123]{Neu69}.

As we will see, the
dynamics on the boundaries of the tiles (edges and vertices)
requires a little bit more elaborated description.

\item As a consequence of the  above results and the study of the dynamics
on the boundary of the level sets, for each map, we easily
characterize the period of every point in terms of the value of its
associate first integral and obtain the global dynamics of the map.
For instance in Proposition \ref{p:algoritme1}  we present our
results in an algorithmic way when $\alpha=\pi/2.$  In particular,
the set of periods of the maps are presented.
\end{enumerate}

\section{Preliminaries and main results}\label{se:pmr}

The families of maps $F$ and $G$ given in \eqref{E:Grho} and
\eqref{e:normalform}, respectively,  are linearly conjugated because
$$F(x,y)=\left(Q^{-1}\cdot G\left(Q\cdot(x,y)^t\right)\right)^t,\quad \mbox{where}\quad Q=\left(\begin{array}{rr}
1&-\cos(\alpha)\\
0& \sin(\alpha)
\end{array}\right).$$

 In this paper we will,
therefore, work with the above normalized one-parameter family  of
maps $F.$  Notice that each map $F$ is bijective with inverse
$$
F^{-1}(x,y)= \left(\begin{array}{rr}
\cos(\alpha)&-\sin(\alpha)\\
\sin(\alpha)&\cos(\alpha)
\end{array}\right)\left(\begin{array}{l}
x\\
y
\end{array}\right) +\left(\begin{array}{l}
\operatorname{sign}(\sin \left( \alpha \right) x+\cos \left( \alpha \right) y)\\
0
\end{array}\right).
$$

Notice also that $F$ is discontinuous in the set $LC_0=\{(x,0):
x\in\R\}.$ We will consider the \emph{critical lines}
$LC_{-i}=\{(x,y)$ such that $F^i(x,y)\in LC_{0}\}$, and also the
\emph{critical set} $\mathcal{F}=\bigcup_{i\in\mathbb{N}} LC_{-i}$
formed by all the preimages of the critical line $LC_0$, where we
use the notation introduced by Mira et al. in \cite{MGBC96} (see
also \cite{AGST,BGM}). We call  the open set
$\mathcal{U}=\R^2\setminus\mathcal{F}$, the \emph{zero-free set}
because none of the orbits starting at point in $\mathcal{U}$
touches the discontinuity line $LC_0,$ where the second coordinate
of the points is zero.

Regarding the above conjugation, of course it is only defined in the
case $\sin(\alpha)\neq 0$ which corresponds with the cases $\rho\ne
\pm 2$. In the case $\alpha=0$ (resp. $\alpha=\pi$) the map $F$ (resp. $F^2$) has trivial dynamics of translation type, and do not correspond
to the initial family $G$ with $\rho=\pm2$.

For each $\alpha\in\{\pi/2,2\pi/3,\pi/3\}$ we introduce some
specific notations and also state our main results: Theorems
\ref{th:a}, \ref{th:b} and \ref{th:c}, respectively. Each one of
them will be proved in a different section. The results in these
theorems have the following structure: in (i) we characterize the
geometry of the critical set; in (ii) we state the existence of a
first  integral in the non-critical set and we characterize the
global dynamics in this set proving that is conjugated with the
composition of two rotations; in (iii) we establish the dynamics in
the critical set; in (iv) we characterize the set of periods of the
maps.

While statements (iv) are known in the literature, the geometric
description given in  statements (i)--(iii) is, as far as we know,
novel.

\subsection{The case $\alpha={\pi}/{2}$} When $\alpha={\pi}/{2}$, the map $F$ is the one in (\ref{E:Grho})
with $\rho=0$  and was studied in
\cite{ChaChe14}. Consider $\mathcal{F}_{\pi/2}$ the grid formed by
the straight lines $x=k$ and $y=\ell$, with $k,\ell\in \mathbb{Z}$.
This grid defines the square  Euclidean regular tiling
\cite{GS77,GS87}, also named \emph{quadrile}, see Figure
\ref{f:T90}. Each (open) tile is denoted by
$$T_{k,\ell}=\{(x,y);\, \mbox{such that }  k<x<k+1\,\mbox{ and }\ell<y<\ell+1\}.$$
The centers of each of these tiles are denoted by
$p_{k,\ell}=\left(k+{1}/{2},\ell+{1}/{2}\right).$ We also introduce
the set
$$\mathcal{U}_{\pi/2}= \bigcup_{(k,\ell)\in\Z^2}
T_{k,\ell}=\R^2\setminus\mathcal{F}_{\pi/2},$$ and the function
\begin{equation}\label{e:Vpi2}
V_{\pi/2}(x,y)=\max \left(  \left|
\operatorname{E}(x)+\operatorname{E}(y)+1 \right| -1, \left|
\operatorname{E}(x)- \operatorname{E}(y) \right|  \right),
\end{equation}
where  $\operatorname{E}(z)=\lfloor z\rfloor$ is the \emph{floor
function} of $z\in\R$ that recall gives as output the greatest
integer less than or equal to $z.$ We also define
 $V_{k,\ell}=V_{\pi/2}(p_{k,\ell})$ and denote $\N_0=\mathbb{N}\cup\{0\}.$
We prove:

\begin{teoa}\label{th:a} Consider the discrete dynamical system (DDS) generated by the
map $F$ given in \eqref{e:normalform} with $\alpha=\pi/2,$
$F(x,y)=(y,-x+\operatorname{sign}(y)).$ Then:
\begin{enumerate}[(i)]

\item Its critical set is  $\mathcal{F}=\mathcal{F}_{\pi/2}.$
\item The function
$V=V_{\pi/2}$ is a first integral of $F$ on the free-zero set
$\mathcal{U}=\mathcal{U}_{\pi/2}.$ Each level set $\{V(x,y)=c\}\cap\mathcal{U},$
with $c\in\N_0,$ is a necklace formed by $4c+2$ squares,
see Figure \ref{f:T90}. If we identify each square with a
point (for instance the center), the DDS restricted to this set is conjugated with the DDS
generated by the map $h:\Z_{4c+2}\to \Z_{4c+2},$ $h(i)=i+c.$ As a
consequence, when $c$ is odd (resp. even), each square in this level
set is invariant by $F^{4c+2}$ (resp. $F^{2c+1}$) and restricted to
this square, $F^{4c+2}$ (resp. $F^{2c+1}$) is a rotation of order
$2$ (resp. $4$), around the center  of the tile. In particular, all
points but the center in each of these tiles have period $8c+4.$

\item All orbits with initial condition on $\mathcal{F}$ are
$(8n+4)$-periodic for some $n\in\N_0,$ see Theorem \ref{t:non
zero-fre epi2} for more details.

\item The map $F$  is
pointwise periodic. Furthermore, its set of periods is
$$
\operatorname{Per}(F)=\left\{4n+1;\,8n+4;\,\mbox{ and }8n+6\mbox{
for all }n\in\N_0\right\}.
$$
\end{enumerate}
\end{teoa}

\begin{figure}[H]
\centerline{\includegraphics[scale=0.7]{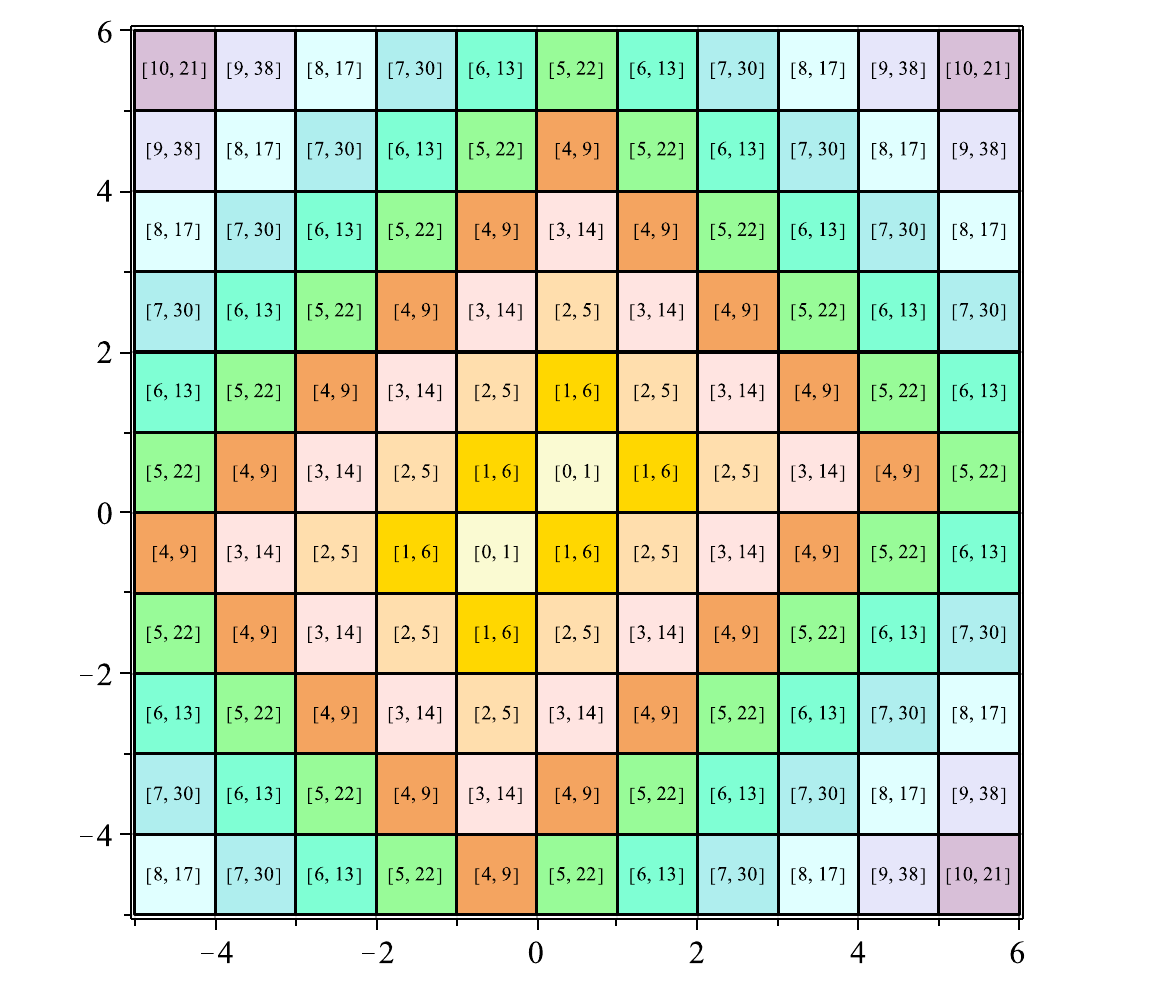}}
\begin{center}

    \caption{Level sets of the first integral $V$ of $F$ for $\alpha={\pi}/{2}$,
     given in \eqref{e:Vpi2}. In each tile $T_{k,\ell}$, the level $V_{k,\ell}$ and the
period of the center $p_{k,\ell}$  are indicated, between brackets.
The other points in the tile have period
$8V_{k,\ell}+4$.}\label{f:T90}
    \end{center}
\end{figure}~
Item $(iv)$ was already proved in \cite{ChaChe14}. All the geometric
description of the dynamics of $F$ given in the other items is new.

Observe that the statement (ii) in the above result can be formalized in the following way: the dynamics of $F$ on each necklace $\{V=c\}\cap\mathcal{U}$
with $c\in\N_0,$ is conjugate with the dynamics of the map
$$
\begin{array}{lccc}
\varphi:&\mathbb{Z}_{4c+2}\times\mathbb{Z}_{q}&\longrightarrow& \mathbb{Z}_{4c+2}\times\mathbb{Z}_{q}\\
&(i,j)&\longrightarrow& (i+c,j+1)
\end{array}
$$ where $q=2$ if $c$ is odd, and $q=4$ if $c$ is even. This map can be seen as the product of two finite order rotations, its first component gives the dynamics on the discrete deferent formed by the set of centers of tiles, and the second component gives the dynamics on a epicycle. A similar situation is described in statements (ii) of Theorems \ref{th:b}
and \ref{th:c}.

Also notice that a simple check shows that the function $V$ is not a first integral
of $F$ on the whole plane, since the relation $V(F)=V$ is not
satisfied for some points in $\mathcal{F}=\R^2\setminus\mathcal{U}$.

As a consequence of the above theorem we can easily give  a simple
algorithm to know the period of each orbit in terms of its initial
condition. Recall that given a point $(x,y)\in\R^2$,
$k=\operatorname{E}(x)$, $\ell=\operatorname{E}(y)$ and
$V_{k,\ell}=V_{\pi/2}(p_{k,\ell}).$ For the forthcoming cases
$\alpha\in\{2\pi/3,\pi/3\},$ from our results a more complicated
algorithm could be obtained, but for the sake of brevity,  we do
not detail it.

%\vfill

\begin{propo}\label{p:algoritme1}
Any point $(x,y)\in\R^2$ is a $p$-periodic point of $F$, where:
\begin{enumerate}[(a)]
\item When $x\notin \Z$ and $y\notin\Z$ and, moreover, either $x-k\neq{1}/{2}$ or $y-\ell\neq{1}/{2}$, then
$p=8V_{k,\ell}+2.$ When  $x-k={1}/{2}$ and $y-\ell={1}/{2}$, then
$p=2V_{k,\ell}+1$ if $V_{k,\ell}$ is even,  and $p=4V_{k,\ell}+2$ if
$V_{k,\ell}$ is odd.

\item When $x\in\Z$ and $y\notin\Z$, if $k$ is even,   $p=8V_{k,\ell}+4$ and  if $k$ is
odd,  $p=8V_{k-1,\ell}+4.$
\item When $x\notin\Z$ and $y\in\Z$, if $\ell$ is even,   $p=8V_{k,\ell}+4$ and if $\ell$ is
odd, $p=8V_{k,\ell-1}+4.$
\item If $x\in\Z$ and $y\in\Z$, then:
\begin{enumerate}[(i)]
\item When $k$ is even,   $p=8V_{k,\ell}+4$ if $\ell$ is even, and
$p=8V_{k,\ell-1}+4$ if $\ell$ is odd.
\item When $k$ is odd,   $p=8V_{k-1,\ell}+4$ if $\ell$ is even, and $p=8V_{k-1,\ell-1}+4$ if $\ell$ is odd.
\end{enumerate}
\end{enumerate}
\end{propo}

The statement (d) in the above result is a consequence of Theorem \ref{t:non zero-fre epi2}. To obtain the result we will identify some tiles, that we will call perfect squares, such that their boundaries  (including their edges and all the vertices in $\mathcal{F}$) avoid the discontinuity effects and, therefore, the points on the boundary of such a tile have the same periodic behavior as the interior points, except their centers. To study the periodicity in the rest of the edges (without vertices) we will associate them with an appropriate tile, so that the points in the edge follow the periodic behavior of the interior points. See Section \ref{ss:dyn-non-free} for more details.

\subsection{The case $\alpha={2\pi}/{3}$}\label{ss:23}

In this case, the map $F$ in \eqref{e:normalform} is conjugate with
the map $G$ in \eqref{E:Grho} with $\rho=1$,  which was studied in \cite{ChaWanChe12}.

We define  $\mathcal{F}_{2\pi/3}$ as the grid formed by the straight
lines $y=\sqrt{3}(x-2k),$ $y=\sqrt{3}\ell$ and
$y=-\sqrt{3}(x-2m-1)$, with $k,\ell,m\in \mathbb{Z}$ and  call
${\mathcal{U}_{2\pi/3}}=\R^2\setminus \mathcal{F}_{2\pi/3}.$ Notice
that ${\mathcal{U}_{2\pi/3}}$ is the (open) trihexagonal Euclidean
uniform tiling (the tessellation 3.6.3.6 in the notation
of~\cite{GS87}), see Figure \ref{f:T120}. In fact, each tile in
${\mathcal{U}_{2\pi/3}}$ is defined by
$$\begin{array}{rl}
T_{k,\ell,m}=&\big\{(x,y),\mbox{ such that }
\sqrt{3}(x-2k-2)<y<\sqrt{3}(x-2k),
\sqrt{3}\ell<y<\sqrt{3}(\ell+1),\\
&   \mbox{ and } -\sqrt{3}(x-2m+1)<y<-\sqrt{3}(x-2m-1)\big\},
\end{array}
$$ with  $m\in\{k+\ell,k+\ell+1,k+\ell+2\},$ where
$k=B(x,y),$ $\ell=C(y)$  and $m=D(x,y),$ being
\[B(x,y)= \operatorname{E} \left( {(3x-\sqrt{3} y)}/{6}\right),\,
C(y)=\operatorname{E} \left({\sqrt{3} y}/{3}\right)\mbox{ and }
D(x,y)= \operatorname{E} \left( {(3x+\sqrt{3} y+3)}/{6}\right).\]
Moreover,
\begin{itemize}
\item   The tile
$T_{k,\ell,k+\ell+1}$ is a regular hexagon  and  its geometric
center (simply center, from now on) is the point
$p_{k,\ell}=\left(2k+\ell+{3}/{2},\sqrt{3}(\ell+{1}/{2})\right);$

\item The tiles $T_{k,\ell,k+\ell}$ and $T_{k,\ell,k+\ell+2}$ are
equilateral triangles  whose respective centers  are
$q_{k,\ell}=\left(2k+\ell+{1}/{2},\sqrt{3}(\ell+{1}/{6})\right)$ and
$r_{k,\ell}=\left(2k+\ell+{5}/{2},\sqrt{3}(\ell+{5}/{6})\right);$
\end{itemize}
and the adherence of the union of the three tiles is a parallelogram
whose sides are $y=\sqrt{3}\ell\,,y=\sqrt{3}(\ell+1)\,,\,y=\sqrt{3}
(x-2k)$ and $y=\sqrt{3}(x-2k-2),$ see Figure \ref{f:Tiles} and Lemma
\ref{l:klm} for more details. Finally, we introduce the function
\begin{multline}\label{e:V2pi3}
V_{2\pi/3}(x,y)=\max \big(\left| B(x,y)-C(y)+ D(x,y)
\right|,\\\left| B(x,y)+C(y)+ D(x,y) +1 \right| -1 , \left| -
B(x,y)+C(y)+ D(x,y) \right| \big).
\end{multline}
Observe that its level sets are discrete and $
\operatorname{Image}(V_{2\pi/3})=\mathbb{N}_0.$ Clearly, $V_{2\pi/3}$
is constant on each tile $T_{k,\ell,m}$ and we denote its value as
\begin{equation}\label{e:Vklm}
V_{k,\ell,m}=\max \left(|k-\ell+m|,|k+\ell+m+1|-1,|-k+\ell+m|
\right).
\end{equation}

\begin{teoa}\label{th:b} Consider the discrete dynamical system  generated by the
map $F$ given in \eqref{e:normalform} with $\alpha=2\pi/3.$ Then:
\begin{enumerate}[(i)]

\item Its critical set is  $\mathcal{F}=\mathcal{F}_{2\pi/3}.$
\item The function
$V=V_{2\pi/3}$ is a first integral of $F$ on the free-zero set
$\mathcal{U}=\mathcal{U}_{2\pi/3}=\R^2\setminus
\mathcal{F}_{2\pi/3}.$
\begin{enumerate}[(a)]
\item  Each level set $\{V(x,y)=c\},$ with $c\in2\N_0$  even, in $\mathcal U$ is a
necklace formed by $6c+2$ triangles, see Figure \ref{f:T120}. If we
identify each triangle with a point (the center, for instance), the DDS restricted to this
set is conjugated with the DDS generated by the map $h:\Z_{6c+2}\to
\Z_{6c+2},$ $h(i)=i+2c.$ As a consequence, each tile  in this level
set is invariant by $F^{3c+1}$  and restricted to this triangle,
$F^{3c+1}$  is a rotation of order $3$ around the center of the
tile.
 In particular, all
points but the center in each of these tiles have period $9c+3.$

\item Each level set $\{V(x,y)=c\},$ with $c\in2\N_0+1$ odd, in $\mathcal U$ is a
necklace formed by $3c+1$ hexagons, see Figure \ref{f:T120}. If we
identify each hexagon with a point, the DDS restricted to this
set is conjugated with the DDS generated by the map $h:\Z_{3c+1}\to
\Z_{3c+1},$ $h(i)=i+c.$ As a consequence, each tile  in this level
set is invariant by $F^{3c+1}$  and restricted to this hexagon,
$F^{3c+1}$  is a rotation of order $3$ around the center of the
tile.
 In particular, all
points but the center in each of these tiles have period $9c+3.$
\end{enumerate}

\item All orbits with initial condition on $\mathcal{F}$ are periodic with period $9n+3$ for some
$n\in\N_0.$

\item The map $F$  is
pointwise periodic. Furthermore, its set of periods is
$$
\operatorname{Per}(F)=\left\{3n+1\mbox{ and } 9n+3\mbox{ for all
}n\in\N_0\right\}.
$$
\end{enumerate}
\end{teoa}~

Similarly to Theorem \ref{th:a}, item $(iv)$ was already known,
see \cite{ChaWanChe12}. Again, all the geometric description of the
dynamics of $F$ given in the other items is new.

From statement (ii), on each necklace, $F$ is conjugate with the product of rotations
$\varphi:\mathbb{Z}_{6c+2}\times\mathbb{Z}_3\righttoleftarrow$  given by $\varphi(i,j)=(i+2c,j+1)$ when $c$ is even, and $\varphi:\mathbb{Z}_{3c+1}\times\mathbb{Z}_3\righttoleftarrow$  given by $\varphi(i,j)=(i+c,j+1)$ when $c$ is odd.

\begin{figure}[H]
\centerline{\includegraphics[scale=0.62]{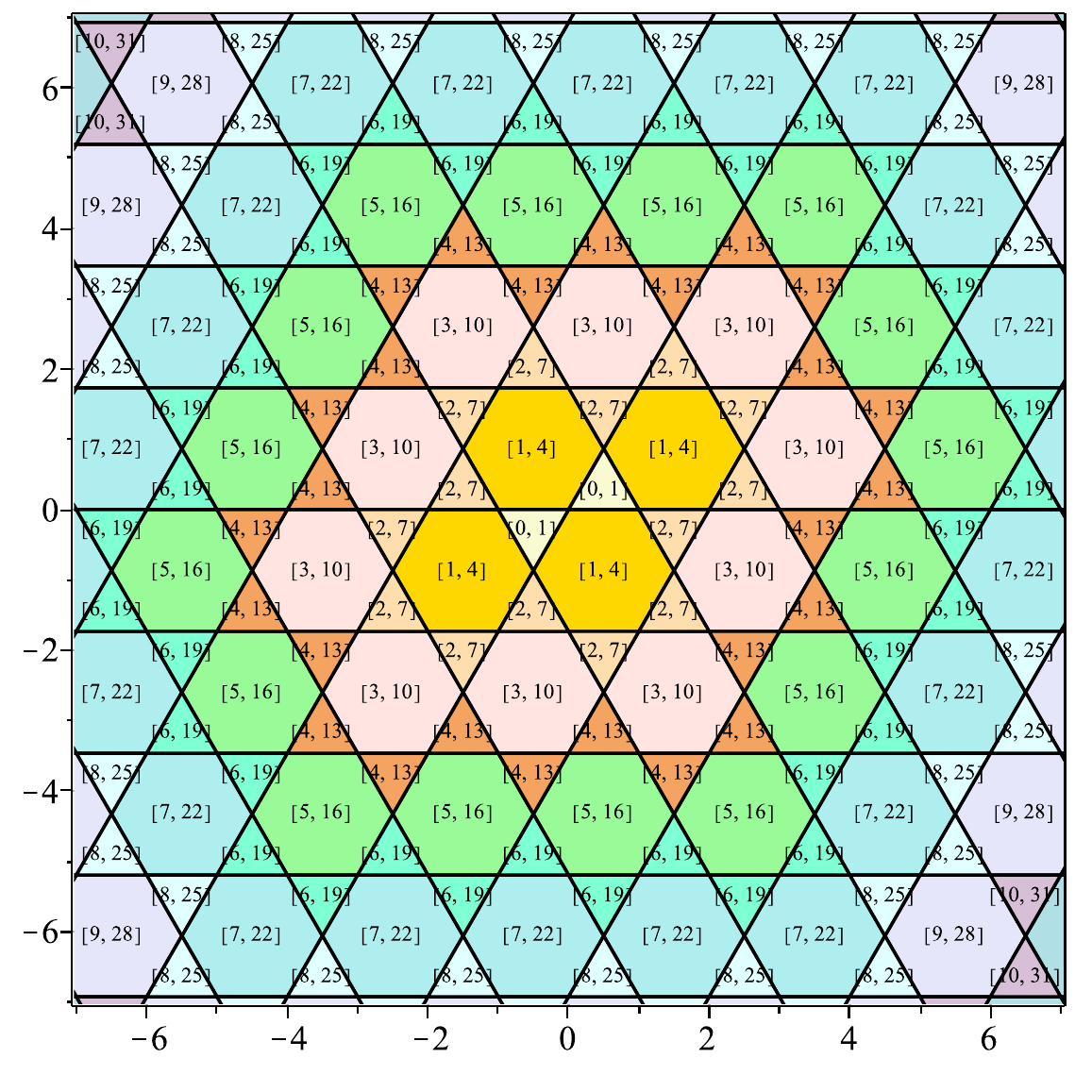}}
    \caption{Level sets of the first integral $V$ given in \eqref{e:V2pi3}. In each tile the level and the period of the
    center are indicated respectively between brackets.}\label{f:T120}
\end{figure}

\subsection{The case $\alpha={\pi}/{3}$}

In this last case, the map $F$ in \eqref{e:normalform} is conjugate
with the map $G$ in \eqref{E:Grho} with $\rho=-1$,  which was
studied in \cite{ChaChe13}.

We consider $\mathcal{F}_{\pi/3}$ the grid formed by the straight
lines $y=\sqrt{3}(x-2k-1)$; $y=\sqrt{3}\ell$ and
$y=-\sqrt{3}(x-2m)$, with $k,\ell,m\in \mathbb{Z}$ that, again, form
a trihexagonal Euclidean uniform tiling which is a translation of
the one that appeared in the previous case $\alpha={2\pi}/{3}$, see
Figure \ref{f:T60}. The interior of each tile is defined by
$$\begin{array}{rl}
T_{k,\ell,m}=&\big\{(x,y),\mbox{ such that }
\sqrt{3}(x-2k-1)<y<\sqrt{3}(x-2k+1),
\sqrt{3}\ell<y<\sqrt{3}(\ell+1),\\
&   \mbox{ and } -\sqrt{3}(x-2m)<y<-\sqrt{3}(x-2m-2)\big\}.
\end{array}
$$

As before, we call ${\mathcal {U}_{\pi/3}}$ the complement of
this grid. Any point $(x,y)\in {\mathcal {U}_{\pi/3}}$ belongs
(only) to the tile $T_{k,\ell,m}$  with $k=B(x,y),$ $\ell= C(y)$ and
$m= D(x,y),$ where \[B(x,y)= \operatorname{E} \left( {(3x-\sqrt{3}
y+3)}/{6}\right),\, C(y)=\operatorname{E} \left( {\sqrt{3}
y}/{3}\right) \mbox{ and } D(x,y)= \operatorname{E} \left(
{(3x+\sqrt{3} y)}/{6}\right)\] and now it can be seen that
$m=k+\ell-1$ or $m=k+\ell$ or $m=k+\ell+1.$  In this case,
\begin{itemize}
\item   The tile
$T_{k,\ell,k+\ell}$ is a regular hexagon  and  its center
is at the point
            $$p_{k,\ell}=\left(2k+\ell+{1}/{2},\sqrt{3}\ell+{\sqrt{3}}/{2}\right).$$

\item The tiles $T_{k,\ell,k+\ell-1}$ and $T_{k,\ell,k+\ell+1}$ are
equilateral triangles  whose centers  are
            $q_{k,\ell}=\left(2k+\ell-{1}/{2},\sqrt{3}\ell+{\sqrt{3}}/{6}\right)$
            and $r_{k,\ell}=\left(2k+\ell+{3}/{2},\sqrt{3}\ell+{5\sqrt{3}}/{6}\right),$
            respectively.
\end{itemize}

We also introduce the following function
\begin{multline}\label{e:Vpi3}
V_{\pi/3}(x,y)=\max \big(\left| B(x,y)-C(y)+ D(x,y) \right|,\left|
B(x,y)+C(y)+ D(x,y) +1 \right| -1 \\, \left| - B(x,y)+C(y)+ D(x,y)+1
\right| -1\big).
\end{multline}
Observe that by construction, it is constant on each tile
$T_{k,\ell,m}.$ Hence we can associate to each point in this tile,
the value $$ V_{k,\ell,m}=\max
\left(|k-\ell+m|,|k+\ell+m+1|-1,|-k+\ell+m+1|-1 \right).$$ Our
results for this case are collected in the next theorem. We remark that
the proof of item~$(iii)$ will be the more complicated part of the
paper.

\begin{teoa}\label{th:c} Consider the discrete dynamical system  generated by the
map $F$ given in \eqref{e:normalform} with $\alpha=\pi/3.$ Then:
\begin{enumerate}[(i)]

\item Its critical set is  $\mathcal{F}=\mathcal{F}_{\pi/3}.$
\item The function
$V=V_{\pi/3}$ is a first integral of $F$ on the free-zero set
$\mathcal{U}=\mathcal{U}_{\pi/3}=\R^2\setminus \mathcal{F}_{\pi/3}.$
\begin{enumerate}[(a)]
\item  Each level set $\{V(x,y)=c\},$ with $c\in2\N_0$  even, in $\mathcal U$ is a
necklace formed by $3c+2$ hexagons, see Figure \ref{f:T60}. If we
identify each one of them with a point, the DDS restricted to this
set is conjugated with the DDS generated by the map $h:\Z_{3c+2}\to
\Z_{3c+2},$ $h(i)=i+c/2.$ As a consequence, when $c=4j$ (resp.
$c=4j+2$),  each tile  in this level set is invariant by
$F^{3c/2+1}$ (resp. $F^{3c+2}$)  and restricted to this hexagon,
$F^{3c/2+1}$ (resp. $F^{3c+2}$) is a rotation of order $6$ (resp.
$3$)  around the center of the tile.
 In particular, all
points but the center in each of these tiles have period $9c+6.$

\item Each level set $\{V(x,y)=c\},$ with $c\in2\N_0+1$ odd, in $\mathcal U$ is a
necklace formed by $6c+4$ triangles, see Figure \ref{f:T90}. If we
identify each one of them with a point, the DDS restricted to this
set is conjugated with the DDS generated by the map $h:\Z_{6c+4}\to
\Z_{6c+4},$ $h(i)=i+c.$ As a consequence, each tile  in this level
set is invariant by $F^{6c+4}$ and restricted to this triangle,
$F^{6c+4}$  is a rotation of order $3$ around the center of the
tile.
 In particular, all
points but the center in each of these tiles have period $18c+12.$
\end{enumerate}

\item All orbits with initial condition on $\mathcal{F}$ are periodic with periods $36n+6,$ $18n+9,$ $18n+15$ or
$108n+72,$ for some $n\in\N_0,$ for more details see Theorem
\ref{t:nonzerofreepi3}.

\item The map $F$  is
pointwise periodic. Furthermore, the set of periods is
\begin{align*}
\operatorname{Per}(F)=&\left\{6n+1;\,12n+8;\,12n+10;\, 18n+9;\,18n+15;\,36n+6;\,36n+24; 36n+30\right.\\
{}&\left.\mbox{ and } 108n+72\mbox{ for all }n\in\N_0\right\}.
\end{align*}
\end{enumerate}
\end{teoa}

Once more, although item $(iv)$ is known, see \cite{ChaChe13}, all
the geometric description of the dynamics of $F$ given in the other
items is new. From the above result, on each necklace, $F$ is conjugate with the map
$\varphi:\mathbb{Z}_{3c+2}\times\mathbb{Z}_q\righttoleftarrow$   
given by $\varphi(i,j)=(i+c/2,j+1)$, where $q=6$ 
when $c\equiv 0\,\operatorname{mod}\,(4)$, and $q=3$ 
when $c\equiv 2\,\operatorname{mod}\,(4)$; or $\varphi:\mathbb{Z}_{6c+4}\times\mathbb{Z}_3\righttoleftarrow$   
where $\varphi(i,j)=(i+c,j+1)$, when $c$ is odd.

\begin{figure}[H]
\centerline{\includegraphics[scale=0.62]{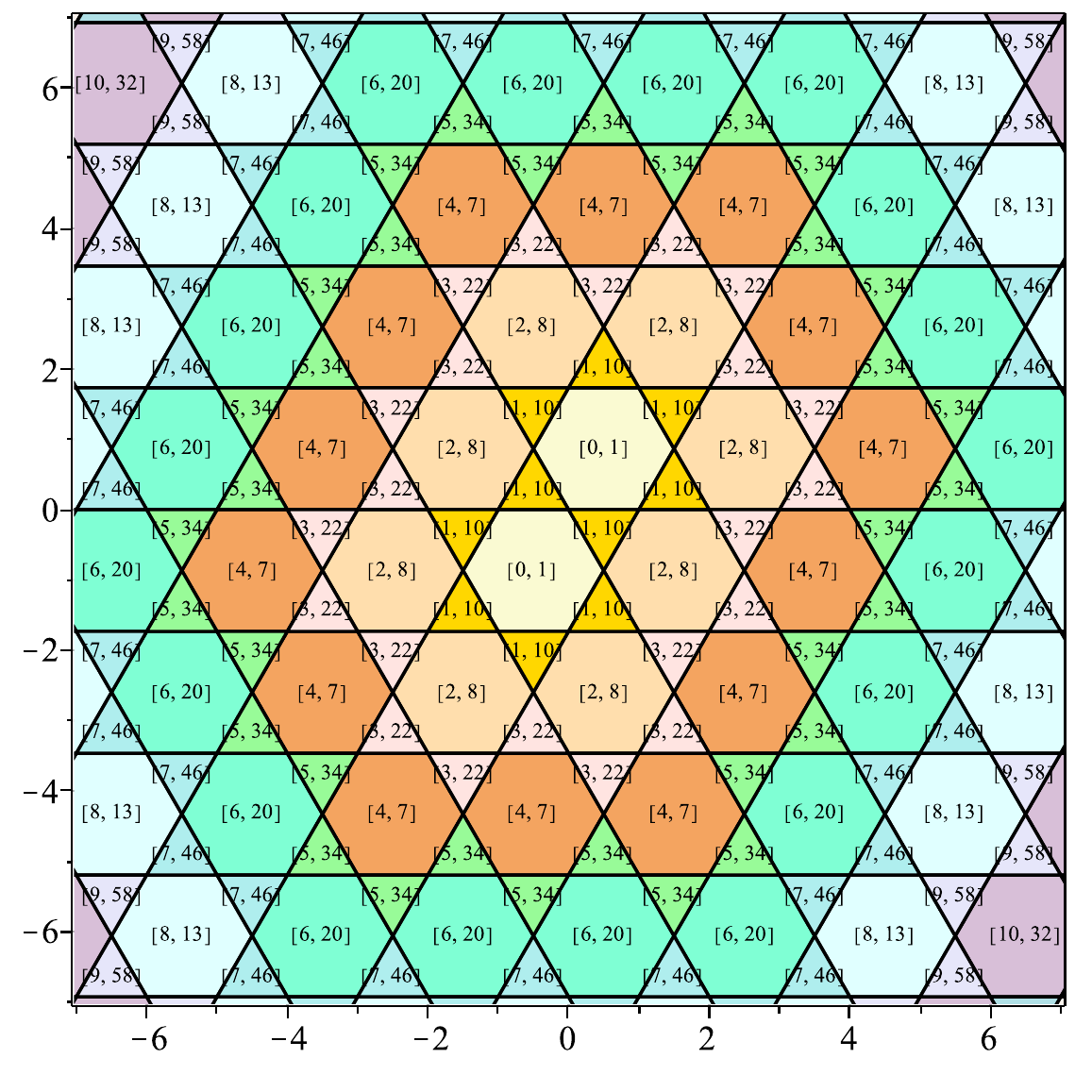}}
    \caption{Level sets of the first integral $V$ of $F$ for $\alpha={\pi}/{3}$, given in in \eqref{e:Vpi3}.
    In each tile the level and the period of the center are indicated respectively between brackets.}\label{f:T60}
\end{figure}

\begin{nota}\label{r:final}
 As we have mentioned, in \cite{GQ} an infinite number of necklaces of a family of maps that include the ones studied in this work are characterized. We want to note that Theorems \ref{th:a}--\ref{th:c} show that for our maps all energy levels are necklaces. In particular, in the case $ \alpha = \pi/2 $ the necklaces studied in \cite{GQ} correspond to the energy levels whose centers have period $4n + 1$ (which are those with even energy level). Let us observe that from Theorem \ref{th:a} we know that there are other necklaces whose period is $4n+2$ (those with odd energy level). In the case $ \alpha = 2 \pi/3 $, the necklaces in \cite{GQ} cover all energy levels since all the necklaces have centers of period $3n + 1$. In the case $\alpha=\pi/ 3$ the  necklaces in \cite{GQ} are those whose centers have period $ 6n+1 $. Observe that Theorem C guarantees the existence of much more necklaces.
\end{nota}

\subsection{The address of a point}%\label{s:prelim}

We end this section with the concept of \emph{address} of a point
that will be used in the proof of Theorems A, B, C.

Recall that every map in the considered one-parameter family $F$ has
discontinuity line  $LC_0=\{y=0\}$, so we introduce the sets
$$H_+=\{(x,y)\in\R^2:y\ge 0\}\,\mbox{ and }\, H_-=\{(x,y)\in\R^2:y< 0\},$$
and we call $F_+$ and $F_-$ the map $F$ restricted to $H_+$ and
$H_-$ respectively. For any point $(x,y)\in \R^2$ we define its
\emph{address} $A(x,y)$ as follows: $$A(x,y)=\left\{
                                              \begin{array}{ll}
                                                +, & \hbox{if $(x,y)\in H_+$,} \\
                                                -, & \hbox{otherwise.}
                                              \end{array}
                                            \right.$$
Moreover for every $n\in\mathbb{N}$, we call the \emph{itinerary of length
$n$} of the point $(x,y)$ the sequence of $n$ symbols
$$\underline I_n(x,y)=(A(x,y),A(F(x,y)),\ldots,A(F^{n-1}(x,y))).$$
Notice that if $\underline I_n(x,y)=(i_1,\ldots,i_n)$ then
$F^n(x,y)=F_{i_n}\circ F_{i_{n-1}}\circ\cdots\circ F_{i_1}(x,y).$

For instance if $(x,y)\in H_+$, $F(x,y)\in H_-$ and $F^2(x,y)\in
H_-$, then the length $3$ itinerary of the point is $\{+,-,-\}$, and
$F^3(x,y)=F_{-}\circ F_{-} \circ F_{+}(x,y).$

\begin{lem}\label{convex} Let $\underline J_n=(i_1,\ldots,i_n)$ be a sequence of
symbols of length $n$ with $i_i\in \{+,-\}$ and consider the set
$B(\underline J_n)= \{(x,y)\in\R^2 \mbox{ such that } \underline
I_n(x,y) =\underline J_n\}$. Then $B(\underline J_n)$ is convex.
Moreover $F^n$ restricted to $B(\underline J_n)$ is an affine map.\end{lem}

\begin{proof} The proof of the convexity follows easily by induction. If $n=1,$
$B(\underline J_n)$ is either $H_+$ or $H_-$ both convex sets.
Assume that the result holds for sequences of length $n-1$ and set
$\underline J_{n-1}=(i_1,\ldots,i_{n-1}).$ Therefore we have
$$B(\underline J_n)=\{(x,y)\in B(\underline J_{n-1}): F^{n-1}(x,y)\in H_{i_n}\}.$$ Moreover,
 $F^{n-1}$ restricted to $B(\underline J_{n-1})$ is the affine map
$G=F_{i_{n-1}}\circ\ldots\circ F_{i_1}.$ So we have
$$B(\underline J_n)=B(\underline J_{n-1})\cap G^{-1}(H_{i_n}).$$ This fact
 proves that $B(\underline J_n)$ is convex because it is
 the intersection of two convex sets. This ends the inductive proof of convexity. Furthermore,
 $F^n(x,y)=F_{i_n}\circ F_{i_{n-1}}\circ\cdots\circ
F_{i_1}(x,y),$ for all $(x,y)\in B(\underline J_n),$ showing that
$F^n$ restricted to $B(\underline J_n)$ is an affine map.
\end{proof}

\section{Proof of Theorem \ref{th:a}}\label{s:alphapi2}

\subsection{Preliminaries}

We start by determining the set of tile centers, $p_{k,\ell}$, such
that $V(p_{k,\ell})=c$ for $c\in\N_0.$ First, we observe that there
are two tiles corresponding to the level set $c=0$. These two tiles
contain the two fixed points of $F$ which are:
$p_{0,0}=\left({1}/{2},{1}/{2}\right)$ that belongs to the tile
$T_{0,0}=(0,1)\times (0,1)$ and
$p_{-1,-1}=\left(-{1}/{2},-{1}/{2}\right)$ that belongs to
$T_{-1,-1}=(-1,0)\times (-1,0).$  It is easy to see that these two
tiles are invariant. To describe the rest of level sets, we denote
$\mathcal{Q}_1=\{(x,y)\in\R^2:x<0,y>0\},$
$\mathcal{Q}_2=\{(x,y)\in\R^2:x>0,y>0\},$
$\mathcal{Q}_3=\{(x,y)\in\R^2:x>0,y<0\},$ and
$\mathcal{Q}_4=\{(x,y)\in\R^2:x<0,y<0\}.$

\begin{lem}\label{l:Nivell} For each level set $\{V=c\}$ with $c\in\N_0$
 there are $4c+2$ centers $p_{k,\ell}$. Furthermore, for each natural number $c\ge 1$ we have:
\begin{enumerate}[(a)]

\item  $\{p_{k,\ell}:V_{k,\ell}=c\}\cap \mathcal{Q}_1=\{(k+{1}/{2},\ell+{1}/{2}):l=k+c,k=-c,-c+1,\ldots ,-1\}.$
We denote by $X_1,X_2,\ldots ,X_c$ these $c$ centers for
$k=-c,-c+1,\ldots ,-1$ respectively. Every one of them lies on the
straight line $y=x+c.$

\item  $\{p_{k,\ell}:V_{k,\ell}=c\}\cap \mathcal{Q}_2=\{(k+{1}/{2},\ell+{1}/{2}):l=-k+c,k=0,1,\ldots ,c\}.$
We denote by $X_{c+1},X_{c+2},\ldots ,X_{2c+1}$ these $c+1$ centers
for $k=0,1,\ldots ,c$ respectively.
 Every one of them lies on the straight line $y=-x+c+1.$

\item  $\{p_{k,\ell}:V_{k,\ell}=c\}\cap \mathcal{Q}_3=\{(k+{1}/{2},\ell+{1}/{2}):l=k-c,k=0,1,\ldots ,c-1\}.$
We denote by $X_{2c+2},X_{2c+3},\ldots ,X_{3c+1}$ these $c$ centers
for $k=c-1,c-2,\ldots ,0$ respectively. Every one of them lies on
the straight line $y=x-c.$

\item  $\{p_{k,\ell}:V_{k,\ell}=c\}\cap \mathcal{Q}_4=\{(k+{1}/{2},\ell+{1}/{2}):l=-k-c-2,k=-1,-2,\ldots,-c-1\}.$
We denote by $X_{3c+2},X_{3c+3},\ldots ,X_{4c+2}$ these $c+1$
centers for $k=-1,-2,\ldots,-c-1$ respectively. Every one of them
lies on the straight line $y=-x-c-1.$
\end{enumerate}
\end{lem}

\begin{proof}%[Proof of Lemma \ref{l:Nivell}]
    In order to prove $(a)$ we begin by considering the points $(k+{1}/{2},k+c+{1}/{2})$ with $-c\le k\le -1.$ Then
    $V_{k,k+c}=\max \left(|2k+c+1|-1,c\right).$ The inequality $-c\le k\le -1$ implies $-c+1\le 2k+c+1\le c-1,$ i.e.
    $|2k+c+1|\le c-1.$ Therefore $|2k+c+1|-1\le c-2$ and consequently $V_{k,k+c}=c.$
    Clearly, $k+{1}/{2}<0$ and $k+c+{1}/{2}> 0$, and hence the points belong to $\mathcal{Q}_1$.

    To see the other inclusion take $(k+{1}/{2},\ell+{1}/{2})\in \mathcal{Q}_1$ with $V_{k,\ell}=c$. We have to prove that
    $\ell=k+c$ and $-c\le k\le -1.$ We know that $k<0,$ $\ell\ge 0$ which easily implies that $\ell>k.$ Hence
     $|\ell-k|=\ell-k$, and $\ell+k<\ell<\ell-k.$ Then
    $$V_{k,\ell}=\max \left(  \left| k+\ell+1\right| -1, \left| k-\ell \right|  \right)=\max \left(  \left| k+\ell+1\right| -1,\ell-k\right).$$
Consider the following two cases:
\begin{enumerate}[(i)]
    \item Assume $\ell+k+1\ge 0.$ Then $V_{k,\ell}=\max(\ell+k,\ell-k)=\ell-k$ because $\ell+k<\ell<\ell-k.$ It implies
     that $\ell-k=c,$ that is $\ell=k+c.$ Furthermore, since $\ell=k+c$ and $\ell\ge 0$ we get $k\ge -c$.
    \item Assume $\ell+k+1< 0.$ Then $V_{k,\ell}=\max(-\ell-k-2,\ell-k).$ Since $(\ell-k)-(-\ell-k-2)=2\ell+2=2(\ell+1)$
     and $\ell+1>0$ also $\ell=k+c$ and the result follows as above.
\end{enumerate}

The proof of statements $(b),(c)$ and $(d)$ follows using the same
easy arguments. ~\end{proof}

In Figure \ref{f:figura1}  we show the points $p_{k,\ell}$ in the
levels $c=2$ and $c=3,$ respectively.

\begin{figure}[H]
\centerline{\includegraphics[scale=0.32]{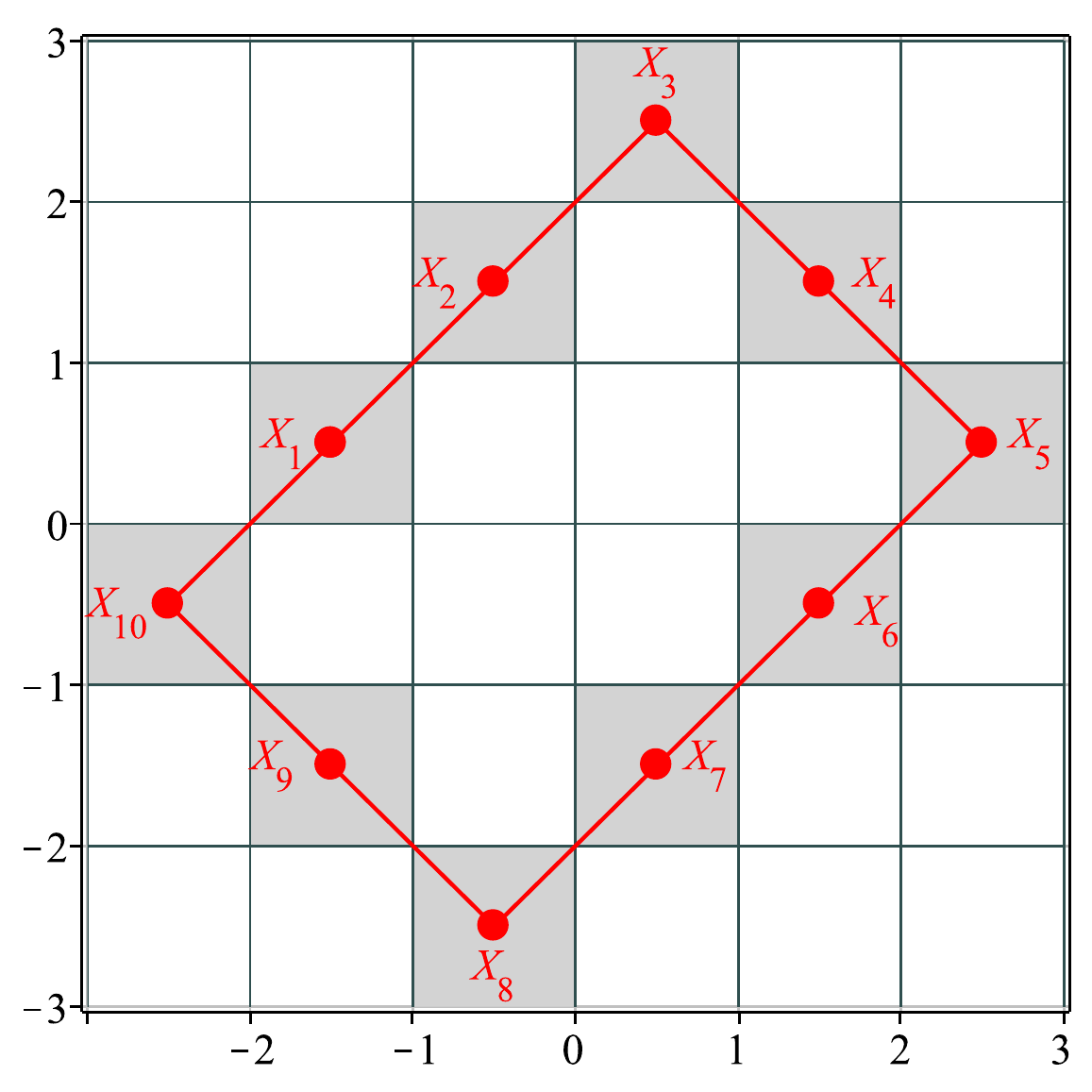}\,\includegraphics[scale=0.32]{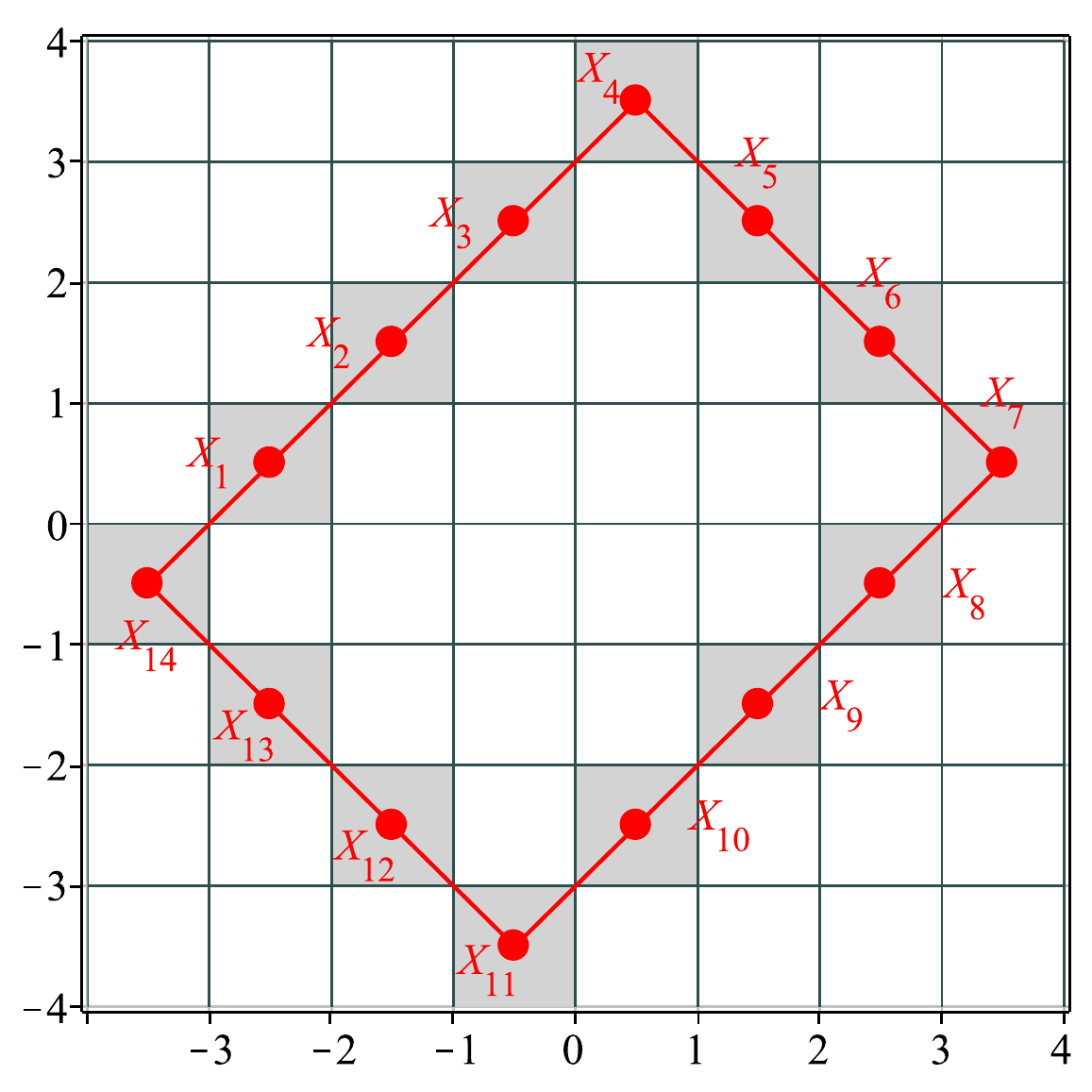}}
        \caption{The centers in the levels $c=2$ and $c=3.$}\label{f:figura1}
\end{figure}

\subsection{Proof of items (i) and (ii) of Theorem \ref{th:a}: dynamics on the zero-free set}

Recall that in this case, $F(x,y)=(y,-x+\operatorname{sign}(y)),$
$F_+(x,y)=(y,-x+1)$ and $F_-(x,y)=(y,-x-1).$ We will split our proof
of items $(i)$ and $(ii)$ of Theorem \ref{th:a} in several lemmas
and propositions.

We start facing the dynamics of the center points of the tiles, or
in other words, the dynamics among the beads of each necklace,
that as we will prove will be invariant under the map $F.$

\begin{lem}\label{l:zetamodul} Fixed  $c\in\mathbb{N}$,  consider the centers $X_1,X_2,\ldots,X_{4c+2}$ which belong to the level
             set $\{V=c\}.$ Then
    \begin{equation}\label{congruencia}
    F(X_i)=X_j \,\,\text{with}\,\,j\equiv i+c\,\, \operatorname{mod}\,\,(4c+2).
    \end{equation}
\end{lem}

\begin{proof}
Consider $i=1,2,\ldots ,c$, then $X_i\in \mathcal{Q}_1.$ From Lemma \ref{l:Nivell} we know that every one of these centers
 is $\left(k+{1}/{2},k+c+{1}/{2}\right)$ with $k=-c,-c+1,\ldots ,-1.$ Since they belong to $H_+$ we have
$$F(X_i)=F_+\left(k+{1}/{2},k+c+{1}/{2}\right)=\left(k+c+{1}/{2},-k+{1}/{2}\right).$$
Denoting $\bar{k}=k+c $ we see that
$\left(k+c+{1}/{2},-k+{1}/{2}\right)=\left(\bar{k}+{1}/{2},
-\bar{k}+c+{1}/{2}\right)$ which satisfies the condition (b) of the
Lemma \ref{l:Nivell}. When $k$
 runs from $-c$ to $-1$ (corresponding with the points  $X_1,X_2,\ldots ,X_c$), then $\bar{k}$ runs from $0$
  to $c-1$ (corresponding with the points  $X_{c+1},X_{c+2},\ldots ,X_{2c}$). Hence we have proved (\ref{congruencia}) for $i=1,2,\ldots, c.$

If $i=c+1$ then $X_{c+1}=\left({1}/{2},c+{1}/{2}\right),$ hence
$F(X_{c+1})=F_+(X_{c+1})=\left(c+{1}/{2},{1}/{2}\right)=X_{2c+1}.$

The proof for $i=c+2,c+3,\ldots, 2c+1$ is done in a similar way, and
also for the rest of values of $i=2c+2,\ldots,4c+2$, but taking into account that in these cases
$F(X_{i})=F_-(X_{i}).$

\end{proof}

As a consequence of Lemma \ref{l:zetamodul} the center points of a
level set form an invariant set and we can prove that the function
$V_{\pi/2}$ defined in~\eqref{e:Vpi2} is a first integral of $F.$

\begin{proof} [Proof  of the first part of item $(ii)$  of Theorem \ref{th:a}] We start
proving that the function $V=V_{\pi/2}$ defined in~\eqref{e:Vpi2} is
a first integral of $F$ on the  set ${\mathcal{U}}={\mathcal
U}_{\pi/2}$.

 Consider a point $(x,y)\in {\mathcal{U}},$ then
$(x,y)\in T_{k,\ell}$ for a certain $k,\ell$ and
  by definition we know that $V(x,y)=V(p_{k,\ell}).$ From Lemma \ref{l:zetamodul} we know that $F(p_{k,\ell})=
p_{\bar{k},\bar{\ell}}$ with $V(p_{k,\ell})=V(p_{\bar{k},\bar{\ell}}).$ On the other hand, since
 each tile is entirely contained in $H_+\setminus \{y=0\}$ or in $H_-,$ $F(T_{k,\ell})=F_+(T_{k,\ell})$
  or $F(T_{k,\ell})=F_-(T_{k,\ell}).$ Since $F_+$ and $F_-$ are rotations (thus isometries)  we get
   that $F$ sends tiles to tiles. In particular $F(T_{k,\ell})=T_{\bar{k},\bar{\ell}}$ and hence
    $V(x,y)=V(p_{k,\ell})=V(p_{\bar{k},\bar{\ell}})=V(F(x,y)).$
\end{proof}

Now we are able to describe the dynamics of the center points and,
in particular, to prove that they are periodic.

\begin{propo}\label{p:periodesdelscentres}
    Every center $p_{k,\ell}$ is a periodic point of $F.$ Furthermore, setting $V_{k,\ell}=c$ we
    have that when $c$ is even (resp. odd), then $p_{k,\ell}$ has period
    $2c+1$ (resp. $4c+2$).
    \end{propo}

\begin{proof} Fix a level $\{V=c\}$ with $c\in\mathbb{N}.$ From Lemma \ref{l:Nivell} we know
 that on $\{V=c\}$ there are $4c+2$ different centers.
    From Lemma \ref{l:zetamodul}, we know that $F$ sends centers to centers, that is, the set $\{X_1,X_2,\ldots, X_{4c+1}\}$
     is invariant by $F.$ Hence, given a center $X_{i_1}$ of the previous set we can study the sequence
    $X_{i_1}\overset{F}{\longrightarrow}X_{i_2}\overset{F}{\longrightarrow}X_{i_3}\overset{F}{\longrightarrow}\cdots .$
    Since the orbit of every center has a finite number of elements and, since $F$ is a bijective map
     and therefore the orbit of $X_{i_1}$ can not be preperiodic, we
    get that $X_{i_p}=X_{i_1}$ for a certain $p$, and therefore it is periodic. Clearly the period must be less or equal to $4c+2.$

 From Lemma \ref{l:zetamodul}, the map $F$ restricted to $\{X_1,X_2,\ldots, X_{4c+1}\}$ is
  conjugate to the map $h:\Z_{4c+2}\longrightarrow \Z_{4c+2}$  defined by $h(i)=i+c.$
    Then
    $$F^p(X_i)=X_i \Leftrightarrow h^p(i)=i \Leftrightarrow i+cp\equiv i \,\, \operatorname{mod}\,\,
    (4c+2)\Leftrightarrow \exists\, n\in\N\mbox{ s.t. } cp=n(4c+2).$$

     Assume that $c=2k$ is an even number. Then $2kp=n(8k+2)\Leftrightarrow kp=n(4k+1).$
         It implies that $p$ is a multiple of $4k+1=2c+1.$ Since $p\le 4c+2$ we get that $p=2c+1$
          or $p=4c+2.$ But we observe that the orbit of $X_i$ only contains some points $X_j$
           with~$j$ having the same parity of $i.$ Hence, we get two different periodic orbits, each one of them of period $2c+1.$

         Assume that $c=2k+1$ is an odd number. Then $(2k+1)p=n(8k+6).$ It implies that~$p$ is a multiple of $8k+6=4c+2.$ Hence $p=4c+2.$

\end{proof}

 We introduce now the concept of \emph{itinerary map associated with a center}.

\begin{defi}\label{d:itinerarymap}
Fix $c\in\mathbb{N}$ and consider one of the centers of the tiles $X_j$ for some $j=1,2,\ldots ,4c+2,$ with $V(X_j)=c.$  Since $X_j$ is
$p$-periodic with $p=2c+1$ or $p=4c+2$ depending on whether $c$ is
even or it is odd,  if we consider its itinerary of length $p$:
$\underline I_p=(i_1,i_2,\ldots,i_p)$ we have that $
F^p(X_j)=F_{i_p}\circ F_{i_{p-1}}\circ\cdots\circ F_{i_1}(X_j)
=X_j.$  We denote this composition by $I_j=F_{i_p}\circ
F_{i_{p-1}}\circ\cdots\circ F_{i_1}$  and we call it the
\emph{itinerary map} associated with $X_j.$
\end{defi}

For instance, if $c=2$  then the center $X_1$ is $5$-periodic and its orbit is
$$X_{1}\overset{F_+}{\longrightarrow}X_{3}\overset{F_+}{\longrightarrow}X_{5}
\overset{F_+}{\longrightarrow}
X_{7}\overset{F_-}{\longrightarrow}X_{9}\overset{F_-}{\longrightarrow}X_{1}$$
This can be easily obtained using the formula \eqref{congruencia} in
Lemma \ref{l:zetamodul} (see also Figure \ref{f:figura1}). Hence
$\underline I_5(X_1)=(+,+,+,-,-)$ and its itinerary map is
$I_1=F_-^2\circ F_+^3$.

When $c=3,$ then using again formula
\eqref{congruencia} in Lemma \ref{l:zetamodul} (see again Figure
\ref{f:figura1}) we get:
\begin{align*}X_{1}&\overset{F_+}{\longrightarrow}X_{4}\overset{F_+}{\longrightarrow}X_{7}
\overset{F_+}{\longrightarrow}
X_{10}\overset{F_-}{\longrightarrow}X_{13}\overset{F_-}{\longrightarrow}X_{2}
\overset{F_+}{\longrightarrow}X_{5}\overset{F_+}{\longrightarrow}
X_{8}\\&\overset{F_-}{\longrightarrow}
X_{11}\overset{F_-}{\longrightarrow}X_{14}\overset{F_-}{\longrightarrow}
X_{3}
\overset{F_+}{\longrightarrow}X_{6}\overset{F_+}{\longrightarrow}X_{9}
\overset{F_-}{\longrightarrow}
X_{12}\overset{F_-}{\longrightarrow}X_{1},
\end{align*}
and hence the itinerary map of $X_1$ is
$I_1=F_-^2\circ F_+^2\circ F_-^3\circ F_+^2\circ F_-^2\circ F_+^3$.

\begin{lem}\label{l:lemarotacio} Fixed $c\in\mathbb{N},$ consider the centers $X_1,X_2,\ldots, X_{4c+2}$
 lying in the level set $\{V=c\}.$ Then for all $j=1,2,\ldots, 4c+2$, the itinerary map $I_j$
  is a rotation centered at $X_j$ of order $4$ if $c$ is even (angle ${\pi}/{2}$), and
   of order $2$ if $c$ is odd (angle $\pi$). In particular $X_j$ is an isolated fixed point of $I_j$.
\end{lem}
\begin{proof} We already know that $I_j(X_j)=X_j.$   We write $$F(x,y)=A \cdot \begin{pmatrix}
 x-\operatorname{sign}(y)\\y \end{pmatrix}\mbox{ where }A=R_{{\pi}/{2}}=\begin{pmatrix}0&1\\-1&0
            \end{pmatrix}.$$

 If  $c=2k$, then by using Proposition \ref{p:periodesdelscentres} we have that
         $X_j$ is $(2c+1)$-periodic, hence, using also that $A^4=\operatorname{Id}$ we obtain
        $$
I_j(x,y)=A^{2c+1}\begin{pmatrix} x\\y \end{pmatrix}+v_j=
A^{4k+1}\begin{pmatrix} x\\y \end{pmatrix}+v_j=A\begin{pmatrix} x\\y \end{pmatrix}+v_j
        $$ for a certain $v_j\in\R^2$. Hence $I_j$ is a rotation of order $4$ centered at $X_j$. Since it has a unique fixed point
        (as $\operatorname{Rank}(A-\operatorname{Id})=2$) then the center of this rotation  is $X_j$.

If  $c=2k+1$,  then $X_j$ has period $4c+2,$ hence using that
$A^2=R_{\pi}=-\operatorname{Id}$, we have:
$$
I_j(x,y)=A^{4c+2}\begin{pmatrix} x\\y \end{pmatrix}+w_j=        A^{8k+6}\begin{pmatrix} x\\y \end{pmatrix}+w_j
=A^{2}\begin{pmatrix} x\\y \end{pmatrix}+w_j,
        $$ for a certain $w_j\in\R^2.$ By using the same argument as before, $I_j$ is a rotation of order~$2$ centered at $X_j$.~\end{proof}

To end  the technical results we establish the next lemma which ensures the all the points
in a tile have the same itinerary  of arbitrary length:

\begin{lem}\label{l:itineraris} All the points in a given tile $T_{k,\ell}$ have the same itinerary of length $n\in\mathbb{N}$ for every $n\in\mathbb{N}.$
\end{lem}
\begin{proof}
Fix $n\in\mathbb{N}$, and suppose that there exist two points $p$
and $q$ in $T_{k,\ell}$ with different itinerary of length $n$ and
let $j\le n-1$ the first time that $A(F^j(p))\ne A(F^j(q)).$ That is
$p$ and $q$ have the same itinerary of length $j$ but $F^j(p)$ and
$F^j(q)$ have different addresses. From Lemma \ref{convex} we have
that all the points in the  segment $\overline{p\,q}$ have the same
itinerary of length $j$  and therefore $F^j$ restricted to
$\overline{p\,q}$ is continuous. Since $F^j(p)$ and $F^j(q)$ have
different addresses it follows that there exists a point $r\in
\overline{p\,q}$ such that $F^j(r)\in LC_0.$ A contradiction because
since  $T_{k,\ell}$ is also convex, $r\in T_{k,\ell}$ and must be
zero-free.
\end{proof}

We can now prove item $(i),$ that is, the zero-free points
are exactly the points $\mathcal{U}=\mathcal{U}_{\pi/2},$ which
belong to the tiles.

\begin{proof}[Proof of item $(i)$ of Theorem \ref{th:a}]
We have already noticed that the zero-free set is included in
${\mathcal{U}_{\pi/2}}.$ Now we are going to see that the boundaries
of the tiles are formed by points which are not zero-free. Consider
a point $p=(k,y)$ where $\ell\le y\le \ell+1$ for a certain
$k,\ell\in\Z.$ Then~$p$ belongs to the right-boundary of
$T_{k-1,\ell}$ and to the left boundary of $T_{k,\ell}.$ Consider
also  the segment $\overline{r s}$ where $r=(k-1/2, y)$ and
$s=(k+1/2, y).$ From Lemma \ref{l:itineraris} the itinerary of any
length of $r$ coincides with the itinerary of the same length of
$p_{k-1,l}$ and the itinerary of any length of $s$ coincides with
the corresponding itinerary of $p_{k,l}.$ Since $p_{k-1,l}$ and
$p_{k,l}$ have different infinite itineraries, there exists $j$ such
that $\underline{I}_{j}(r)=\underline{I}_{j}(s)$ but $A(F^j(r))\ne A(F^j(s).$ Now from Lemma
\ref{convex} it follows that there exists $t\in \overline{rs}$ such
that $F^j(t)\in LC_0.$ Clearly this point must be $p.$

If we consider a point which belongs to a horizontal boundary of two
consecutive tiles, then its iterate belongs to a vertical
boundary of two consecutive tiles and then we can apply the above
result.
\end{proof}

\begin{proof}[Continuation of the proof of item $(ii)$ of Theorem \ref{th:a}] Consider the tile $T_{k,\ell}$. Let
 $X_j$ be its center. By Proposition \ref{p:periodesdelscentres} it is a $p$-periodic
point, and by Lemma \ref{l:itineraris}  all the points in the tile
have the same itinerary of length $p$, hence $
\left.F^p\right|_{T_{k,\ell}}=I_j. $ Moreover, by Lemma
\ref{l:lemarotacio}, on each tile $I_j$ is a  rotation centered at
$X_j$ of the order established in the statement.

  Assume that $V_{k,\ell}=c$ with $c$ an even number.
 We have already proved that each center $X_j$ in this level set has period $2c+1$ (see
  Proposition \ref{p:periodesdelscentres}). The points in the orbit of $X_j$ are points
   $X_i$ with $i$ having the same parity of $j$ (see Lemma \ref{l:zetamodul}). Hence
    we have two orbits, the first one formed by $X_1,X_3,\ldots,X_{4c+1}$ and the second one formed by $X_2,X_4,\ldots, X_{4c+2}$, and
     therefore we know the period of the centers of the tiles.

The period of all the points of the tiles $T_{k,\ell}$ but the
centers is a consequence that we have proved that $
\left.F^p\right|_{T_{k,\ell}}=I_{p_{k,\ell}}, $ where
$I_{p_{k,\ell}}$ is the itinerary map of $p_{k,\ell}$, which is a
rotation of order $4,$ see again Lemma \ref{l:lemarotacio}.

When $V_{k,\ell}=c$ with $c$ an odd number the proof is similar.
\end{proof}

\subsection{Proof of item $(iii)$ of Theorem \ref{th:a}: dynamics on the non zero-free set}\label{ss:dyn-non-free}

Following the notation introduced in Lemma \ref{l:Nivell}, for any fixed energy level $c$ of $V,$
 there are $4c+2$ tiles  with centers $X_1,X_2,\ldots,X_{4c+2}$. Let us denote $T_j$ to the tile with center $X_j$.
Also, for a fixed energy level  $c$, we denote by $Q_j$ the \emph{closed} square formed by the tile $T_j$
and its boundary, that is $Q_j=T_j\cup\partial T_j$.

For each energy level $c$ even, we will call the squares
$Q_1,Q_3,\ldots, Q_{4c+1}$ \emph{perfect squares} because, as we will see,
these closed squares evolve avoiding the discontinuity effects of $F.$

 Clearly every edge of a square is also an edge of the consecutive square. The perfect squares are positioned
 as Figure \ref{f:nonzerofree2} displays,  the perfect squares being the red ones.
~\begin{figure}[h]
   \centerline{\includegraphics[scale=0.4]{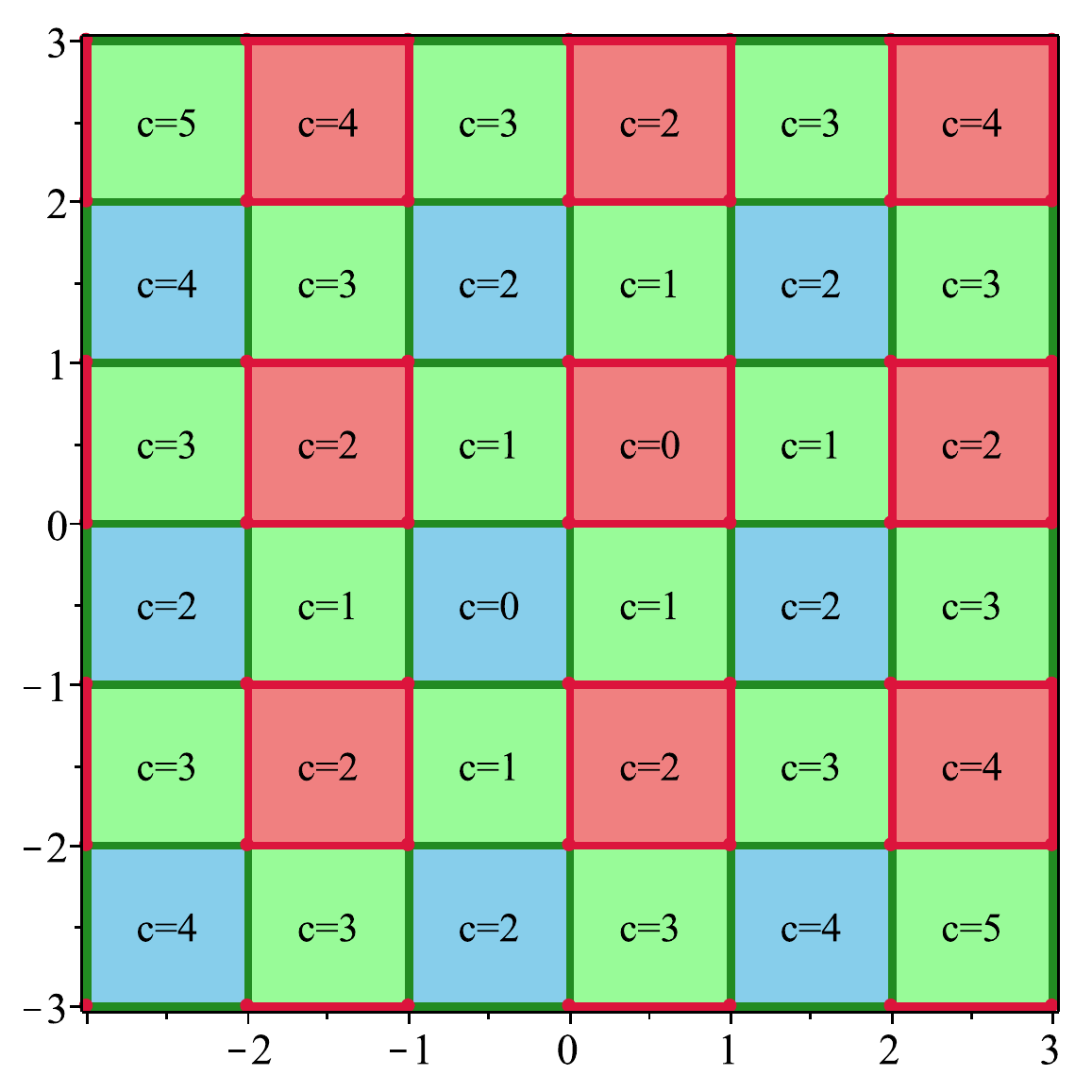}}
   \caption{Position of the perfect squares in red.  The green tiles correspond to odd energy
    levels and the blue ones correspond to the even energy levels such that their corresponding squares are not perfect.
    The borders and the vertices are highlighted with the color of the tile having the same itinerary map. The level set $\{V=c\}$ is indicated in each square.}\label{f:nonzerofree2}
\end{figure}

From the above figure we see that it is enough to prove the periodicity of
 the points on the boundary of the perfect squares and the periodicity of
  the points on the boundary on the squares of odd levels which are not in the boundary of the perfect squares.

Our result also will ensure that the points of the border (including the vertices)
 of any perfect square are periodic with the same period as the points of the tile
   corresponding to the perfect square (excluding the center). Observe that any
    vertex point in $\mathcal{F}$  belongs uniquely to a perfect square. Hence the result will characterize the dynamics of all the vertices.
     The rest of the non-zero points are periodic with the same period as the points of the adjacent tile (excluding the center)
      with odd energy levels. See again Figure \ref{f:nonzerofree2}.

Item $(iii)$ of Theorem \ref{th:a} is a straightforward consequence
of the next theorem.

\begin{teo}\label{t:non zero-fre epi2}   Consider the level set $V=c$.
        \begin{enumerate}[(a)]
    \item If $c$ is an even number, then every point on the boundary of the squares $Q_1,Q_3,\ldots, Q_{4c+1}$ is  a $(8c+4)$-periodic point.

    \item    If $c$ is an odd number, then when $j$ is odd (resp. even) the two horizontal (resp. vertical) edges of $Q_j,$
     without the vertices, are formed by $(8c+4)$-periodic points.
    \end{enumerate}
    \end{teo}

Prior to proving the result we stress the following fact:
\begin{nota}\label{r:dinamicaquadrats}
On every point in $H_+$ the map $F=F_+$. Hence, for all
$j=1,2,\ldots ,2c+1,$ we have that $F(Q_j)=F_+(Q_j)=Q_{i}$ with
$i\equiv j+c$ mod $4c+2$  (see Lemmas \ref{l:zetamodul} and
\ref{l:itineraris}). Analogously, for $i=2c+3,\ldots,4c+1$ we have
$F(Q_j)=F_-(Q_j)=Q_{i}$ with $i\equiv j+c$ mod $4c+2$, since these
squares are contained in $H_-$. \emph{Observe, however, that the
situation for the squares $Q_{2c+2}$ and $Q_{2c+4}$ is quite
different because on the \emph{top} edge of these two squares
$F=F_+$ while in the rest of the square $F=F_-$.} In  Figure
\ref{f:nonzerofree}, we display the position of the tiles
corresponding with the centers $X_1$, $X_{2c+1}$, $X_{2c+2}$ and
$X_{4c+2}$ with respect to the discontinuity line $LC_0$.
    \end{nota}~
\begin{figure}[H]
   \centerline{\includegraphics[scale=0.5]{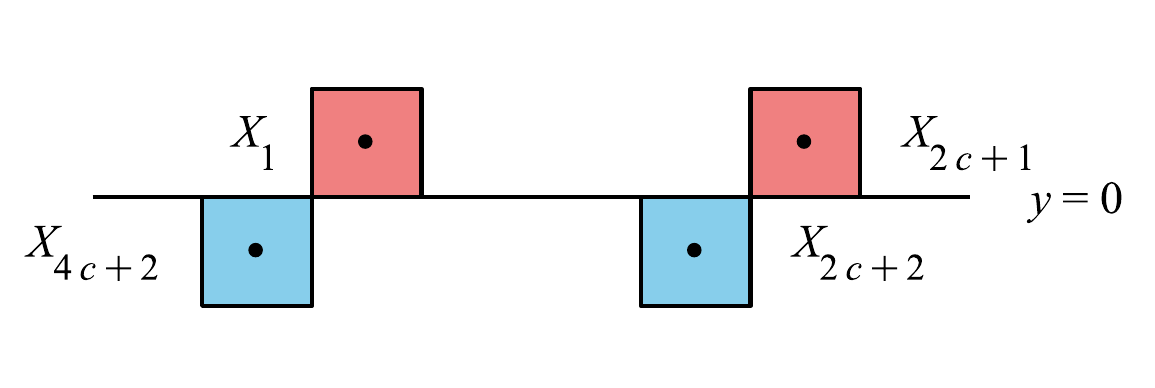}}
    \caption{Position of the squares $Q_1$, $Q_{2c+1}$, $Q_{2c+2}$ and $Q_{4c+2}$, together with its centers,
     for any level set $V=c.$}\label{f:nonzerofree}
\end{figure}

\begin{proof}[Proof of Theorem \ref{t:non zero-fre epi2}]
    Consider the squares $Q_1,Q_3,\ldots, Q_{4c+1}$ for $c$ even, that is, the perfect squares on this level.
     The only squares in this particular collection $Q_j$ with $j$ odd which intersect $y=0$ are $Q_1$ and $Q_{2c+1}.$
      But from Remark \ref{r:dinamicaquadrats} we know that $F(Q_j)=Q_{k}$ with $k\equiv j+c$ mod $4c+2$ for all $j$ odd,
       including the cases with $j=1$ and $j=2c+1.$ In particular this implies that this set of squares is invariant.
        In consequence, by continuity, the points in the boundary of $Q_j$ inherit the dynamics of the points in
         $T_j\setminus X_j$ and, therefore, they are periodic with period $4(2c+1)$. Furthermore, $\left.F^{2c+1}\right|_{Q_j}$ is a rotation of order 4.

    Now assume that $c$ is odd. We notice that the squares $Q_1,Q_3,\ldots ,Q_{4c+1}$
     (resp. $Q_2,Q_4,$ $\ldots,Q_{4c+2}$) share every vertical (resp. horizontal)
      edge with an edge of a perfect square, which we already know is periodic. Hence
       we need to follow the dynamics of their horizontal (resp. vertical) edges or, in other
       words, the dynamics of $\widetilde{Q}_j=Q_j\setminus \{$ its vertical edges, including the vertices$\}$
        (resp. $\widetilde{Q}_j=Q_j\setminus \{$its horizontal edges, including the vertices$\}$).

    Since now $X_1,X_2,\ldots ,X_{4c+2}$ belong to the same periodic orbit,
     the set of corresponding squares contains $\widetilde{Q}_{2c+2}$ and $\widetilde{Q}_{4c+2}.$
The result will be proved if we can ensure the invariance of the set
of squares $\widetilde{Q}_j$.  In order to do this, we must ensure
that the edges we are studying are not pre-images of the top edges
of the squares $Q_{2c+2}$ and $Q_{4c+2}.$  So, first, we study for
which values of $p,$ $F^p(X_j)=X_{2c+2}$ or $F^p(X_j)=X_{4c+2}.$
    \begin{itemize}
        \item From Lemma \ref{l:zetamodul} we have $F^p(X_j)=X_{2c+2}$ if and only if $j+pc\equiv 2c+2$ mod
        $4c+2$, that is if there exists  $n\in\N$ such that $j+pc=2c+2+n(4c+2).$ Hence $j+pc$ is an even
        number and since $c$ is odd we get that $p$ and $j$ have the same parity.
        \item Analogously, $F^p(X_j)=X_{4c+2}$ if and only if $j+pc\equiv 0$
        mod  $4c+2$, which means that there exists $ n\in\N$ such that $ j+pc=n(4c+2).$ As before $p$ and $j$ have the same parity.
    \end{itemize}
Assume that $j$ is odd,  and let the $p$-iterate of $Q_j$ be the
first one that  reaches $Q_{2c+2}$ (or $Q_{4c+2}$). Since it is the
first time that the images of  $Q_j$ intersect $\{y=0\}$, we can
still apply the arguments in the proof of  Lemma \ref{l:lemarotacio}
and therefore $F^p|_{Q_j}$ is an even-order rotation. Thus, since
$p$ is also odd, the horizontal edges of $Q_j$ are mapped via $F^p$
to the vertical edges of $Q_{2c+2}$ (or $Q_{4c+2}$). Therefore
$F^p(\widetilde{Q}_j)=\widetilde{Q}_{2c+2}$ (or
$\widetilde{Q}_{4c+2}$). It implies that for all $j$ odd,
$F(\widetilde{Q}_j)=\widetilde{Q}_{j+c}.$

The same arguments work when $j$ is even: $p$ is even too and $F^p$
sends the vertical edges of $Q_j$ to the vertical  edges of
$Q_{2c+2}$ (or $Q_{4c+2}$). So also
$F(\widetilde{Q}_j)=\widetilde{Q}_{j+c}$ for every $j$ even.

The condition $F(\widetilde{Q}_j)=\widetilde{Q}_{j+c}$ implies that,
by continuity, the points in the edges under study of $Q_j$
inherit the dynamics of the points in  $T_j\setminus X_j$, which are
periodic with period $2(4c+1)$.
\end{proof}

Consider the square $Q_{k,\ell}=T_{k,\ell}\cup\partial T_{k,\ell}$,
and set $c=V_{k,\ell}$; then, we will say that the  square has
odd label (resp. even label) if $Q_{k,\ell}=Q_j$ for an odd  value
$j$ (resp. even) in the order introduced in Lemma \ref{l:Nivell}.
Observe that, in particular, with the proof of Theorem~\ref{t:non
zero-fre epi2} we also have proved the following result that gives
the dynamics of $F$ on all the points in~$\mathcal{F}$.

\begin{corol}\label{t:teo4}
Consider the square $Q_{k,\ell}=T_{k,\ell}\cup\partial T_{k,\ell}$,
and set $c=V_{k,\ell}$, then:
\begin{enumerate}[(a)]
\item If $c$ is even, and the square has odd label then $Q_{k,\ell}$
 is invariant under the action of the map $F^{2c+1}$, and $\left.F^{2c+1}\right|_{Q_{k,\ell}}$ is a rotation of order $4$
 centered at $p_{k,\ell}$; as a consequence the edges of these $Q_{k,\ell}$ are formed by $4(2c+1)-$periodic points.

\item If $c$ is odd, and the square has odd label (resp. even label) then the horizontal  (resp. vertical)
edges (excluding the vertices) are invariant under the action of the
map $F^{4c+2}$ which is also  a  rotation of order $2$ centered at
$p_{k,\ell}$ on that edges (excluding the vertices); as a consequence the edges of these squares which are not edges of a perfect square are formed by $2(4c+2)-$periodic points.

\end{enumerate}
\end{corol}

Remember that  if the square has even energy level and even label we
treat their boundaries as being  part of the boundary of the
adjacent odd-energy level tile.

\subsection{Proof of item $(iv)$ of Theorem \ref{th:a}}

The proof simply follows by collecting the results of the previous
items.

\section{Proof of Theorem \ref{th:b}}\label{s:alpha2pi3}

\subsection{Preliminaries}

For each tile $T_{k,\ell,m},$ we start determining $m$ in terms of
$k$ and $\ell.$

\begin{lem}\label{l:klm}
Any point $(x,y)\in \mathcal{U}={\mathcal{U}_{2\pi/3}}$ belongs to a
tile $T_{k,\ell,m}$ where either, $m=k+\ell$ or     $m=k+\ell+1$ or
$m=k+\ell+2.$
\end{lem}
\begin{proof}
From the inequalities $ k< {x}/{2}-{\sqrt{3}y}/{6}<k+1$ and $\ell<
{\sqrt{3} y}/{3}<\ell+1$ it follows that
${1}/{2}+k+\ell<{x}/{2}+{\sqrt{3} y}/{6}+{1}/{2}<k+\ell+{5}/{2}.$
Hence $m=\operatorname{E}({x}/{2}+{\sqrt{3} y}/{6}+{1}/{2})$ is
either $m=k+\ell,$ $m=k+\ell+1$ or $m=k+\ell+2.$
\end{proof}
In fact, the set of points satisfying $ k<
{x}/{2}-{\sqrt{3}y}/{6}<k+1$ and $\ell< {\sqrt{3} y}/{6}<\ell+1$
form a parallelogram whose sides are
$y=\sqrt{3}\ell\,,y=\sqrt{3}(\ell+1)\,,\,y=\sqrt{3} (x-2k)$ and
$y=\sqrt{3}(x-2k-2).$ From the proof of the above lemma we see that
\begin{enumerate}[(a)]
    \item $m=k+\ell \Leftrightarrow k+\ell+{1}/{2}<{x}/{2}+{\sqrt{3}y}/{6}+{1}/{2}<k+\ell+1 \Leftrightarrow
    \sqrt{3}(-x+2k+2\ell)<y<\sqrt{3}(-x+2k+2\ell+1)$
        \item $m=k+\ell+1 \Leftrightarrow k+\ell+1<{x}/{2}+{\sqrt{3}y}/{6}+{1}/{2}<k+\ell+2\Leftrightarrow\sqrt{3}
        (-x+2k+2\ell+1)<y<\sqrt{3}(-x+2k+2\ell+3)$
        \item $m=k+\ell+2 \Leftrightarrow k+\ell+2<{x}/{2}+{\sqrt{3}y}/{6}+{1}/{2}<k+\ell+\frac{5}{2}\Leftrightarrow\sqrt{3}
        (-x+2k+2\ell+3)<y<\sqrt{3}(-x+2k+2\ell+4)$
\end{enumerate}
Denoting by
$L_1=\{y=\sqrt{3}(-x+2k+2\ell)\},L_2=\{y=\sqrt{3}(-x+2k+2\ell+1)\},L_3=\{y=\sqrt{3}(-x+2k+2\ell+3)\}$
and $L_4=\{y=\sqrt{3}(-x+2k+2\ell+4)\},$ we can draw its graphics in
Figure \ref{f:Tiles}.

\begin{figure}[H]
\centerline{\includegraphics[scale=0.45]{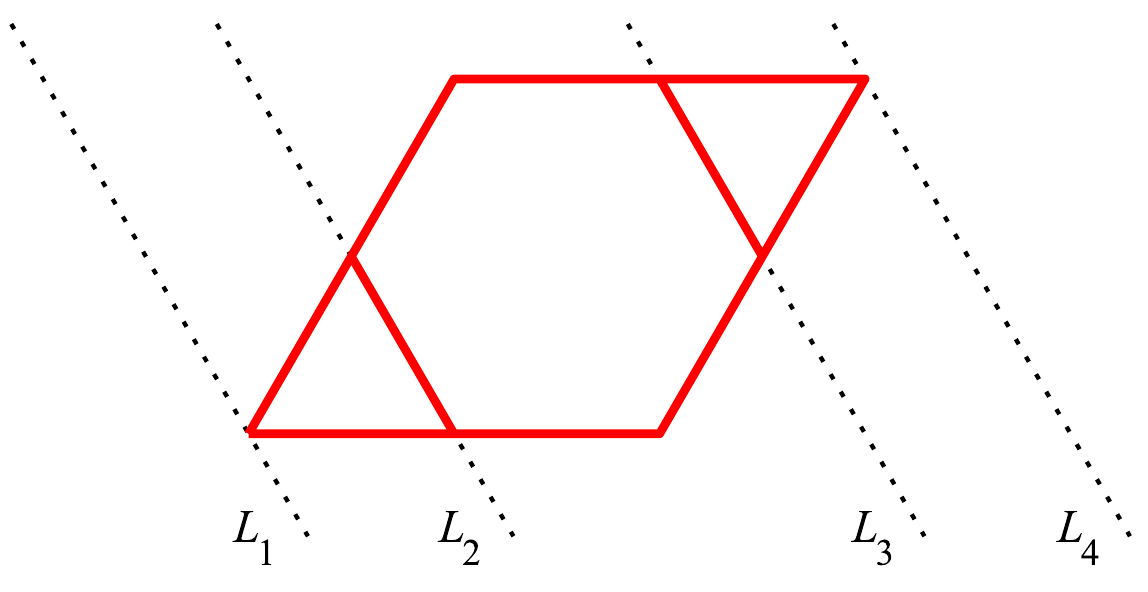}}
    \caption{The tiles  $T_{k,\ell,k+\ell}\cup T_{k,\ell,k+\ell+1}\cup T_{k,\ell,k+\ell+2}$    when $\alpha={2\pi}/{3}.$}\label{f:Tiles}
\end{figure}

In Figure \ref{f:Tiles}, we can see that the parallelogram has been
partitioned into three sets: two triangles and one hexagon. Thus, we
have proved the following lemma, where we use the notation
introduced in Section \ref{ss:23}.

\begin{lem}\label{l:centroid1}
    Let $T_{k,\ell,m}$ be one tile of  $\mathcal{U}_{2\pi/3}=\mathcal{U}$.
    \begin{enumerate}[(a)]
        \item If $m=k+\ell$ or $m=k+\ell +2$, then $T_{k,\ell,m}$
        is a triangle whose center is either $q_{k,\ell}$ or  $r_{k,\ell},$
        respectively.
        \item If $m=k+\ell +1$ then $T_{k,\ell,m}$
        is a hexagon whose center  is
        $p_{k,\ell}.$
    \end{enumerate}
\end{lem}

\subsection{Proof of items $(i)$ and $(ii)$: dynamics on the zero-free set}

As in the case $\alpha=\pi/2$ we split the proof of these two items
into several lemmas and propositions. Here there is an added
difficulty, there are tiles with hexagonal shape and others with
triangular shape. We study them separately.

\subsubsection{Dynamics on the hexagonal tiles. The case $m=k+\ell+1$} From Lemma \ref{l:centroid1}
the tile $T_{k,\ell,k+\ell+1}$ is a regular hexagon. Set
$V_{k,\ell}$ for the value of $V$ on  $T_{k,\ell,k+\ell+1}$. Then $
V_{k,\ell}=V_{k,\ell,k+\ell+1}=\max
\left(|2k+1|,2|k+\ell+1|-1,|2\ell+1| \right). $ The next two results
characterize the set of centroids of the hexagons, that is their
number and  geometric locus, for this case.
\begin{lem}\label{l:hexagons}
    Let $m=k+\ell+1$ and
    $p_{k,\ell}=\left(2k+\ell+{3}/{2},\sqrt{3}\left(\ell+{1}/{2}\right)\right).$
    Then
    \begin{enumerate}[(a)]
        \item $V(p_{k,\ell})=c$ is an odd number.
        \item The set $\{p_{k,\ell}:V_{k,\ell}=c\}$ has $3c+1$ points. In particular there are $3c+1$
        hexagons $T_{k,\ell,k+\ell+1}$  in this energy level.
        \item The points $p_{k,\ell}$ with $V_{k,\ell}=c$ lie in the irregular hexagon determined by the intersection
        of the straight lines $y=\sqrt{3}(x\pm c)$, $y=\pm {\sqrt{3}c}/{2}$ and $y=\sqrt{3}(-x\pm(1+c)),$ see Figure \ref{f:figurahexagonc3}.
    \end{enumerate}
\end{lem}

\begin{figure}[H]
    \centerline{\includegraphics[scale=0.40]{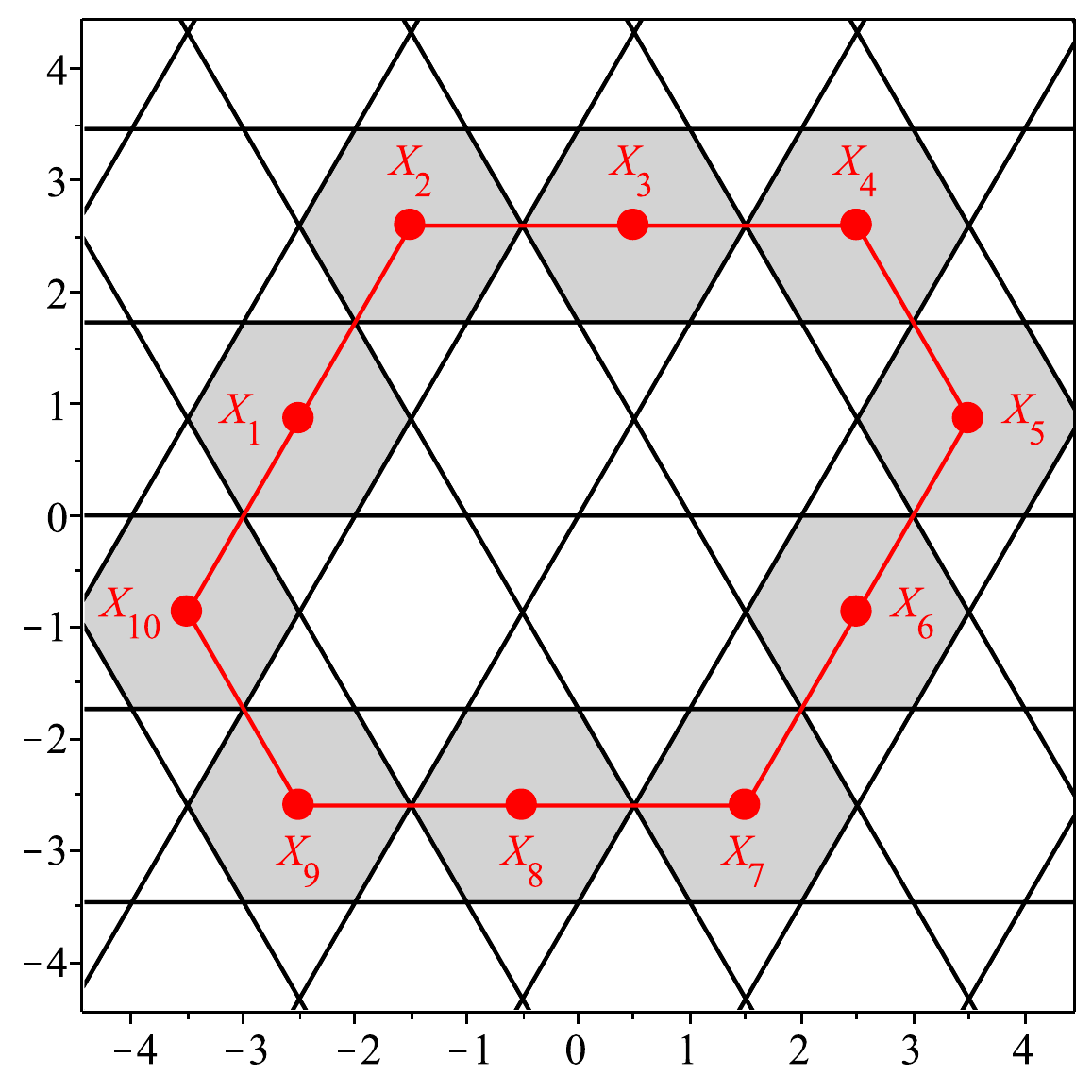}}
    \caption{The centers  points $p_{k,\ell}$  in the level $c=3.$}\label{f:figurahexagonc3}
\end{figure}

\begin{proof} Since $V_{k,\ell}$ depends on the signs of $2k+1,$ $2\ell+1$ and $k+\ell+1,$ we are going to consider the different cases.
    \begin{enumerate} [(1)]
        \item Assume that $k<0$, $l\ge 0$ and $k+\ell+1\le 0.$ From \eqref{e:Vklm} we have that  $V_{k,\ell}=
        \max\left(-2k-1,2\ell+1,-2k-2\ell-3\right),$ and from the inequalities
        $(2\ell+1)-(-2k-1)\le 0$  and  $(-2k-2\ell-3)-(-2k-1)<0$
        we get that $V_{k,\ell}=-2k-1=c.$ Hence $c$ is odd, $k=-(1+c)/{2}$, $0\le \ell\le -k-1={(c-1)}/{2},$ and the
        points $p_{k,l}$ can be written as
$
p_{k,l}=\left(-c+\ell+{1}/{2},\sqrt{3}\left(\ell+{1}/{2}\right)\right)$
with  $0\le \ell \le {(c-1)}/{2}.$ Hence, every $p_{k,l}$ lies on
        the straight line $y=\sqrt{3}(x+c)$ and that there are ${(c+1)}/{2}$ 
        such points.
        \item   Assume that $k<0$, $l\ge 0$ and $k+\ell+1> 0.$ Then, proceeding analogously to the previous case we get
        $V_{k,\ell}=2\ell+1=c.$ Hence $c$ is odd, $\ell={(c-1)}/{2}$ and ${-(c+1)}/{2}<k<0.$ We get the points
        $
        p_{k,l}=\left(2k+{(c+2)}/{2},{\sqrt{3}c}/{2}\right)$ with  ${(1-c)}/{2}\le k \le -1.
        $
        These points lie on the straight line $y={\sqrt{3}c}/{2}$ and there are ${(c-1)}/{2}$ such points.

        \item Assume that $k\ge 0$ and $\ell\ge 0.$ Then $V_{k,\ell}=2k+2\ell+1=c>0.$ Hence $c$ is odd,
        $k={(c-1-2\ell)}/{2}$ and $0\le \ell\le {(c-1)}/{2}.$ We get the points
        $
    p_{k,l}=\left(c-\ell+{1}/{2},\sqrt{3}\left(\ell+{1}/{2}\right)\right)$ with  $0\le \ell \le {(c-1)}/{2}.
    $
    These points lie on $y=\sqrt{3}(-x+1+c)$ and there are ${(c+1)}/{2}$ such points.

    \item Assume that $\ell< 0,$  $k\ge 0$ and $k+\ell +1\ge 0.$ Then  $V_{k,\ell}=2k+1=c.$ Hence $c$ is odd, $k={(c-1)}/{2}$
    and ${-(c+1)}/{2}\le \ell<0.$ We obtain the points
    $
    p_{k,l}=\left(c+\ell+{1}/{2},\sqrt{3}\left(\ell+{1}/{2}\right)\right)$  with  ${-(c+1)}/2\le \ell\le -1,
    $
     which lie on $y=\sqrt{3}(x-c)$ and there are ${(c+1)}/{2}$ such  points.

     \item Assume that $\ell< 0,$  $k\ge 0$ and $k+\ell +1< 0.$ Then  $V_{k,\ell}=-2\ell-1=c.$ Hence $c$ is odd, $\ell={-(c+1)}/{2}$
     and $0\le k<{(c-1)}/{2}.$ We get the points
     $
     p_{k,l}=\left(2k+{(2-c)}/{2},-{\sqrt{3}c}/{2}\right)$  with  $0\le k\le {(c-3)}/{2},
     $
     which lie on $y=-{\sqrt{3}c}/{2}$ and there are ${(c-1)}/{2}$ such  points.
     \item Assume $k<0$ and $\ell<0.$ Then $k+\ell+1<0$ and $V_{k,\ell}=-2k-2\ell-3=c.$ Hence $c$ is odd, $k=-\ell-{(c+3)}/{2}$
     with $-{(c+3)}/{2}<\ell<0.$ We get the points
     $
     p_{k,l}=\left(-c-\ell-{3}/{2},\sqrt{3}\left(\ell + {1}/{2}\right)\right)$ with  $-{(c+1)}/{2}\le \ell\le -1,
     $
which lie on $y=-\sqrt{3}(x+c+1)$ and there are ${(c+1)}/{2}$ 
such points.
\end{enumerate}
The Lemma follows from the above case-by-case study.
\end{proof}

Consider an \emph{odd} energy level $V_{k,\ell}=c$; we will
label the center points $p_{k,\ell}$  analogously as in the case
$\alpha={\pi}/{2}$: we denote by $X_1$ the point on the
corresponding irregular  hexagon defined by the lines in the above
lemma, which belongs to $H_+$ and its first component is the
smallest one; that is, $X_1=(-c+{1}/{2},{\sqrt{3}}/{2}).$ After we
denote by $X_2,X_3,\ldots X_{3c+1}$ the consecutive points on the
hexagon turning clockwise (see Figure \ref{f:figurahexagonc3} for
instance). The set of center points in such  a level set is
invariant under the action of $F$ and its dynamics is given in the
next result:

\begin{propo} \label{p:zmodul2} Assume $m=k+\ell+1$. Fixed
 $V_{k,\ell}=c$ an odd number, consider the points $X_1,X_2,\ldots, X_{3c+1}$ introduced above. Then
\begin{enumerate}[(a)]
    \item For all $i=1,2,\ldots ,3c+1$ , $F(X_i)=X_j$ with $j\equiv i+c \text { mod }(3c+1).$
    \item The set $\{X_1,X_2,\ldots, X_{3c+1}\}$ is a periodic orbit of period $3c+1.$ \end{enumerate}
\end{propo}
\begin{proof} To prove statement $(a)$ we are going to consider the points that are on each of the six sides of the irregular hexagon
delimited by the straight lines in Lemma \ref{l:hexagons}.
    \begin{itemize}
    \item Consider the points $X_1,X_2,\ldots ,X_{{(c+1)}/{2}}$ which lie on $y=\sqrt{3}(x+c).$ The map $F_+$ sends this
    straight line to $y=\sqrt{3}(-x+1+c)$ and
        $$F(X_1)=F_+\left(-c+{1}/{2},{\sqrt{3}}/{2}\right)=\left({(c+2)}/{2},{\sqrt{3}c}/{2}\right)=X_{c+1}.$$
        Since the distance between two consecutive points is constant and $F_+$ is an isometry we get that $X_2,X_3,\ldots
        ,X_{{(c+1)}/{2}}$  are mapped to $X_{c+2},X_{c+3},\ldots ,X_{{(3c+1)}/{2}}$ respectively. In particular
        \begin{equation}\label{e:auxiliar}
        F(X_i)=X_{i+c}
        \end{equation}
for all $i=1,2,\ldots, X_{{(c+1)}/{2}}.$
        \item Now consider $X_{{(c+1)}/{2}},X_{{(c+3)}/{2}},\ldots, X_{c+1}$ which lie on $y={\sqrt{3}c}/{2}.$ $F_+$ sends
        $y={\sqrt{3}c}/{2}$ to the straight line $y=\sqrt{3}(x-c)$ and we already check that $X_{{(c+1)}/{2}}$ is
        mapped to $X_{{(3c+1)}/{2}}.$ Being $F_+$ an isometry we get that $X_{{(c+3)}/{2}},\ldots ,X_{c+1}$ are
        mapped to $X_{{(3c+3)}/{2}},\ldots ,X_{2c+1}$ respectively. Finally, equation \eqref{e:auxiliar} also holds for
        every $i={(c+1)}/{2},\ldots ,c+1.$
        \item Following the same argument it is seen that $X_{c+1},X_{c+2},\ldots,X_{{(3c+1)}/{2}}$ are mapped
        to $X_{2c+1},X_{2c+2},\ldots,X_{{(5c+1)}/{2}}$ respectively.
        \item The points $X_{{(3c+3)}/{2}},\ldots,X_{2c+1}$ lie on $y=\sqrt{3}(x-c)$ and on $H_-.$ Then we have
        to take into account $F_-$ which maps $y=\sqrt{3}(x-c)$ to $y=-\sqrt{3}(x+c+1)$ and
        $$F(X_{{(3c+3)}/{2}})=F_-\left(c-{1}/{2},-{\sqrt{3}}{2}\right)=\left({-(c+2)}/{2},-{\sqrt{3}c}/{2}\right)=
        X_{{(5c+3)}/{2}}.$$ Hence $F_-$ sends $X_{{(3c+5)}/{2}},\ldots,X_{2c+1}$ to $X_{{(5c+5)}/{2}},\ldots,X_{3c+1}$
        respectively and \eqref{e:auxiliar} holds for every $i={(3c+3)}/{2},{(3c+5)}/{2},\ldots, 2c+1.$

        \item Now consider the points of the irregular hexagon which lie on the straight line $y=-{\sqrt{3}c}/{2},$ that is, $X_{2c+1},\ldots,X_{{(5c+3)}/{2}}.$
        The map $F_-$ sends  $y=-{\sqrt{3}c}/{2}$ to $y=\sqrt{3}(x+c)$ and we verify that $X_{2c+1}$ is sent to $X_{3c+1}.$
        Then $X_{2c+2},\ldots,X_{\frac{5c+3}{2}}$ are sent to $X_1,\ldots,X_{{(c+1)}/{2}}.$ So $F(X_i)=X_{j}$ with $j\equiv i+c \text { mod }(3c+1)$
        for $i=2c+1,\ldots,{(5c+3)}/{2}.$
        \item Finally, the points $X_{{(5c+5)}/{2}},\ldots,X_{3c+1}$ are on $H_-$ and also lie on $y=-\sqrt{3}(x+c+1)$.
         The map $F_-$ sends this straight line to $y={\sqrt{3}c}/{2}$ and since we already know that that $F_-(X_{{(5c+5)}/{2}})=
         X_{{(c+1)}/{2}}$ we get that $X_{{(5c+5)}/{2}},\ldots,X_{3c+1}$ are sent to $X_{{(c+3)}/{2}},\ldots,X_c$ respectively.
         Again $F(X_i)=X_{j}$ with $j\equiv i+c \text { mod }(3c+1)$ for every $i={(5c+5)}/{2},\ldots,X_{3c+1}$.
        \end{itemize}
    In order to prove $(b)$ we proceed as in the proof of Proposition \ref{p:periodesdelscentres}. We use that
    the map $F$ restricted to $\{X_1,X_2,\ldots, X_{3c+1}\}$ is conjugated to
$h:\Z_{3c+1}\longrightarrow \Z_{3c+1}$ defined by $h(i)=i+c.$
    Then
    $$F^p(X_i)=X_i \Leftrightarrow h^p(i)=i \Leftrightarrow i+cp\equiv i \,\, \text{mod}\,\,(3c+1)\Leftrightarrow \exists n\in\N:cp=n(3c+1).$$
    This implies that $p$ must be a multiple of $3c+1,$ and since $p\le 3c+1$ we get that the minimal period is $p=3c+1$ as we wanted to see.
    \end{proof}

\subsubsection{Dynamics on the triangular tiles. The cases $m=k+\ell$ and $m=k+\ell+2$}
For the triangular tiles, a result analogous to Lemma
\ref{l:hexagons} is the following. We omit all the details.
\begin{lem}\label{l:triangles}
\begin{enumerate}[(i)]
        \item Let $m=k+\ell$ and  the points $q_{k,\ell}=\left(2k+\ell+{1}/{2},\sqrt{3}(\ell+{1}/{6})\right).$ Then
        \begin{enumerate}[(a)]
            \item $V(q_{k,\ell})=c$ is an even number.
            \item The set $\{q_{k,\ell}:V(q_{k,\ell})=c\}$ has $3c+1$ elements. In particular there are $3c+1$  triangles $T_{k,\ell,k+\ell}$
            in this energy level.
            \item The points $q_{k,\ell}:V(q_{k,\ell})=c$ lie in the irregular hexagon determined by the intersection of the six straight lines
            $y=\sqrt{3}(x+c-{1}/{3}),y={\sqrt{3}}(3c+1)/6,y=-\sqrt{3}(x-c-{2}/{3}),y=\sqrt{3}(x-c-{1}/{3}),y=
            -{\sqrt{3}}(3c-1)/6$ and $y=-\sqrt{3}(x+c+{4}/{3}).$
        \end{enumerate}
    \item Let $m=k+\ell+2$ and the points $r_{k,\ell}=\left(2k+\ell+{5}/{2},\sqrt{3}(\ell+{5}/{6})\right).$ Then
    \begin{enumerate}[(a)]
        \item $V(r_{k,\ell})=c$ is an even number.
        \item The set $\{r_{k,\ell}:V(r_{k,\ell})=c\}$ has $3c+1$ elements. In particular there are $3c+1$  triangles $T_{k,\ell,k+\ell}$
        in this energy level.
        \item The points $r_{k,\ell}:V(r_{k,\ell})=c$ lie in the irregular hexagon determined by the intersection of the six straight lines
        $y=\sqrt{3}(x-c+{1}/{3}),y={\sqrt{3}}(3c-1)/6,y=-\sqrt{3}(x+c+{4}/{3}),y=\sqrt{3}(x-c+{1}/{3}),
        y=-{\sqrt{3}}(3c+1)/6$ and $y=-\sqrt{3}(x+c-{4}/{3}).$
    \end{enumerate}
\end{enumerate}
    \end{lem}

\begin{figure}[H]
\centerline{\includegraphics[scale=0.37]{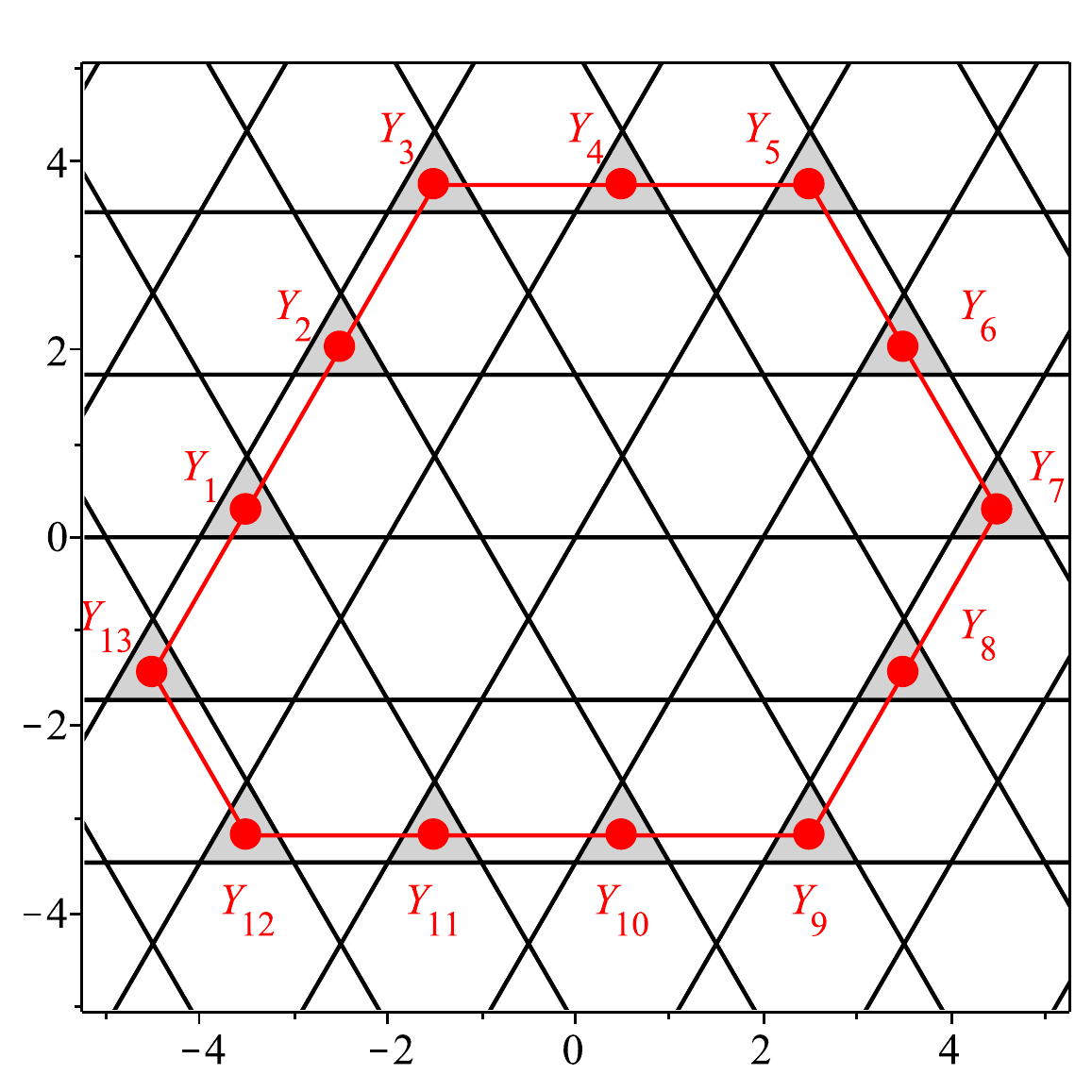}\,
\quad \includegraphics[scale=0.37]{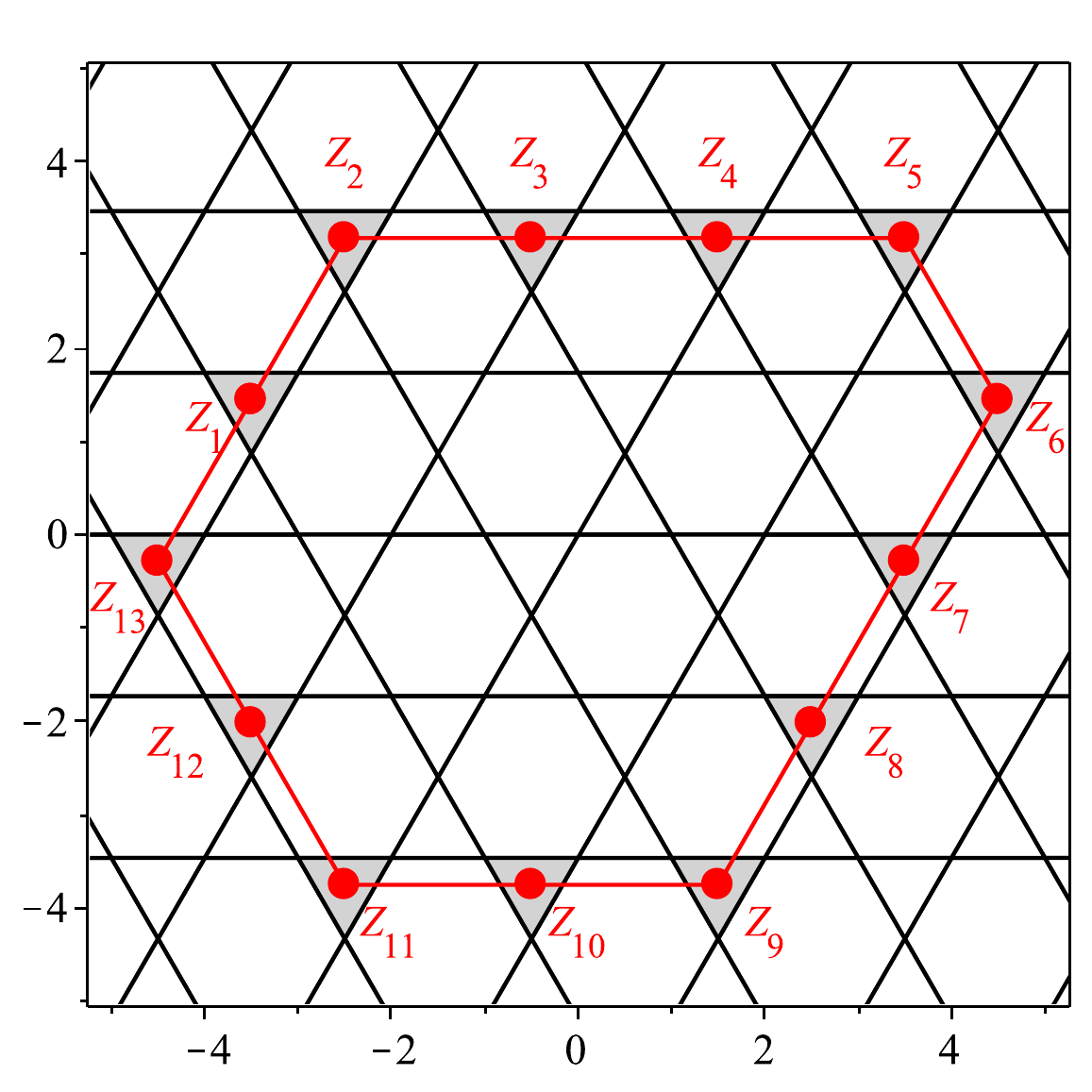}}
    \caption{Position of the centers $q_{k,\ell}$ and $r_{k,\ell}$ in the level  $c=4.$}\label{f:hexc4}
\end{figure}

For a fixed even number $c\ge 2$ consider
$\{q_{k,\ell}:V(q_{k,\ell})=c\}$ (resp.
$\{r_{k,\ell}:V(r_{k,\ell})=c\}$). We denote by $Y_1$ (resp. $Z_1$)
the point $q_{k,\ell}$ (resp. $r_{k,\ell}$) in the corresponding
irregular hexagon defined by the lines in the above  lemma, which
belongs to $H_+$ and its first component is the smallest one, that
is $Y_1=\left(-c+{1}/{2},{\sqrt{3}}/{6}\right)$ (resp.
$Z_1=\left(-c+{1}/{2},{5\sqrt{3}}/{6}\right)$). We denote by
$Y_2,Y_3,\ldots,Y_{3c+1}$ (resp. $Z_2,Z_3,\ldots,Z_{3c+1}$) the
consecutive points on the corresponding hexagon turning clockwise.
See Figure~\ref{f:hexc4} where the center points of the energy level $c=4$ are shown.

\begin{propo}\label{p:centretriangles} Consider a fixed even energy level $V_{k,\ell}=c$ and  the
points $Y_1,Y_2,\ldots ,Y_{3c+1}$ and $Z_1,Z_2,\ldots ,Z_{3c+1}$
defined above. Then
    \begin{enumerate}[(a)]
    \item  For all $i=1,2,\ldots ,3c+1,$ $F(Y_i)=Y_j$ with $j\equiv i+c \text{ mod }(3c+1)$ and $F(Z_i)=Z_j$
    with $j\equiv i+c\text{ mod } (3c+1).$
    \item The set $\{Y_1,Y_2,\ldots ,Y_{3c+1}\}$ is a periodic orbit of period $3c+1$ and the set $Z_1,Z_2,\ldots ,Z_{3c+1}$
    also is a periodic orbit of period $3c+1.$
        \end{enumerate}
\end{propo}

The proof follows exactly by the same arguments involved in the
proofs of  Lemma \ref{l:zetamodul} and
Proposition~\ref{p:periodesdelscentres}. The next corollary simply
consists of gluing (in a suitable way) the two sets given in the
previous proposition, to form a single necklace with $6c+2$
triangular beads.

\begin{corol}\label{co:nou} Consider a fixed even energy level $V_{k,\ell}=c$ and
denote the set of $6c+2$ ordered points $Y_1,Z_1,Y_2,Z_2\ldots
,Y_{3c+1},Z_{3c+1},$ that we will denote $W_{i},i=1,2,\ldots,6c+2.$
Then  for all $i=1,2,\ldots ,6c+2,$ $F(W_i)=W_j$ with $j\equiv i+2c
\text{ mod }(6c+2).$
\end{corol}

\begin{proof}[Proof of item $(i)$ of Theorem \ref{th:b}] Following the spirit of Definition \ref{d:itinerarymap}, we
can introduce the concept of itinerary map for the  centers
$p_{k,\ell}$, $r_{k,\ell}$ and $q_{k,\ell}$ in an analogous way.
Then, the proof is exactly the same proof as for item $(i)$
of Theorem \ref{th:a}. It is based on the fact that all the points
in the same tile have the same itineraries of arbitrary length (a
result analogous to Lemma \ref{l:itineraris}) and also on Lemma
\ref{convex}.
\end{proof}

\begin{proof}[Proof of item $(ii)$ of Theorem \ref{th:b}] We start proving that
$V=V_{3\pi/2}$ is a first integral.  As in the proof of item $(ii)$
of Theorem \ref{th:a},  we notice that since the tiles are
completely contained in $H_+\setminus\{y=0\}$ or $H_-$ and the maps
$F_{\pm}$ are rotations, then $F$ sends tiles to tiles. Remember that by
its definition $V$ is constant on each tile, and in particular takes
the value attained at the center point. The result follows now from
the fact that in each level set, the set of centers is invariant,
see Propositions \ref{p:zmodul2} and \ref{p:centretriangles}.

 Similarly that in the proof of Theorem \ref{th:a} we consider the tile $T_{k,\ell,m}$.
We know that all the points in the tile have the same itinerary than
its center which, by Propositions~\ref{p:zmodul2}, \ref{p:centretriangles} and Corollary
\ref{co:nou} give the discrete dynamical systems generated by the functions
$h$ given in the statement of Theorem \ref{th:b} between the
corresponding $\Z_M.$ Moreover, we know that the centers are
periodic with period $3c+1.$ Hence, if $I$ is the itinerary map
associated with the center point, that is
$I=\left.F^{3c+1}\right|_{T_{k,\ell,m}},$ we have that
$I(T_{k,\ell,m})=T_{k,\ell,m}$. Writing $F(x,y)=A \cdot
\left(x-\operatorname{sign}(y),y \right)^t$ where
$A=R_{{2\pi}/{3}}$, we have $I=A^{3c+1}+v=A+v$ for a certain
$v\in\R^2,$ $v\ne 0,$ which implies that $I$ is a rotation with a unique
fixed point, hence it is the center point.
Furthermore $I^3=\operatorname{Id}$ because
$I^3=A^3+(A^2+A+\operatorname{Id})v=\operatorname{Id},$  since
$A^2+A+\operatorname{Id}=0.$  In summary, $I$ is a rotation of angle
$\alpha={2\pi}/{3}$ which implies that the points $(x,y)\in
T_{k,\ell,m}$ which are not centers
 are 3-periodic for $F^{3c+1}$ and consequently, they are $(9c+3)$-periodic.~\end{proof}

\subsection{Proof of item $(iii)$  of Theorem \ref{th:b}: dynamics in the non zero-free set}

From the previous results, we know that the non zero-free set
$\mathcal{F}$ is  formed by the borders of the tiles, both hexagons
and triangles.

Consider an energy level $c\in\mathbb{N}_0$. Assume that $c$ is an
odd number, then the level  set $\{V=c\}$ is formed by $3c+1$
hexagonal tiles, whose centers $X_1,X_2,\ldots ,X_{3c+1}$ form a
periodic orbit. Denoting by $H_i$ the closure of this hexagon we
also know that $H_1$ and $H_{2c-1}$ intersect $y=0$ at the bottom
edge while $H_{2c}$ and $H_{3c+1}$ intersect $y=0$ at the top edge.
When $c$ is even, we have the points $Y_1,Y_2,\ldots ,Y_{3c+1}$
(respectively, $Z_1,Z_2,\ldots ,Z_{3c+1}$). Each $Y_i$ (resp. $Z_i$)
is the center of an upward (resp. downward) facing triangle; its
closure intersects $y=0$ only when $i=1$ and $i={(3c+2)}/{2}$ (resp.
$i={(3c+2)}/{2}$ and $i=3c+1$), see Figure \ref{f:nonzerofreed1d2}.

\begin{figure}[H]
    \centerline{\includegraphics[scale=0.4]{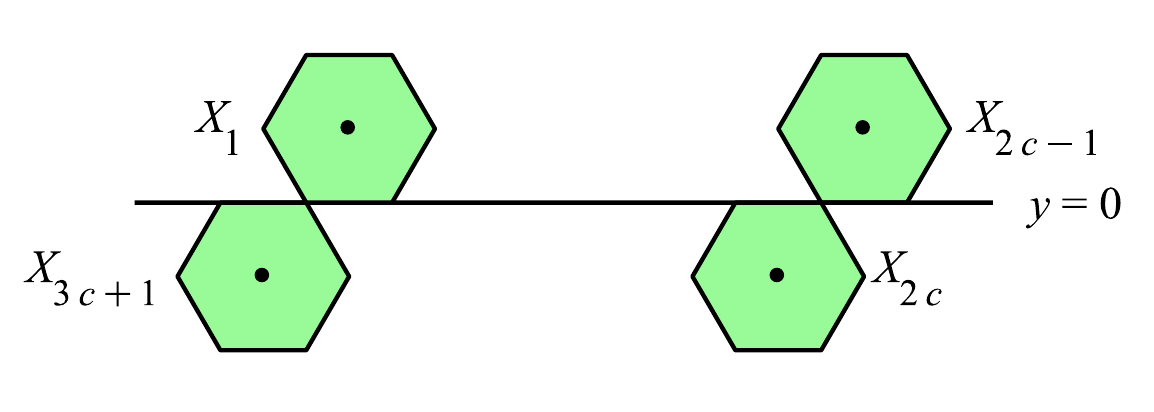}}
    \vspace{0.3cm}
    \centerline{\includegraphics[scale=0.4]{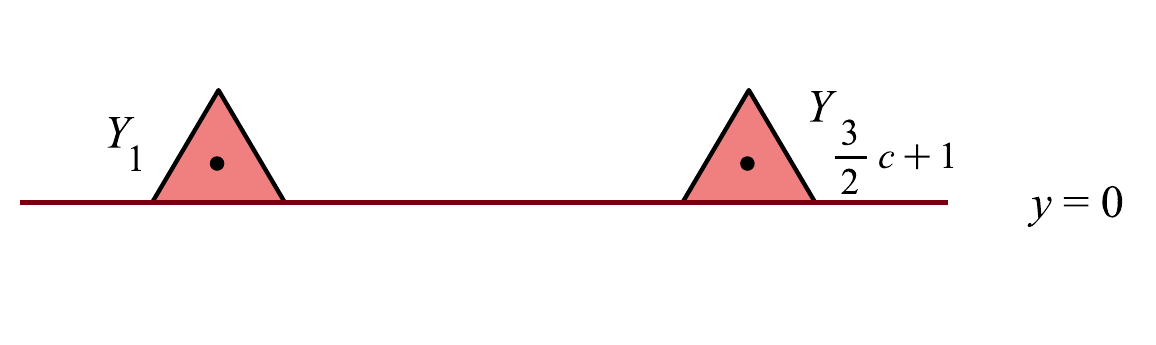} \includegraphics[scale=0.4]{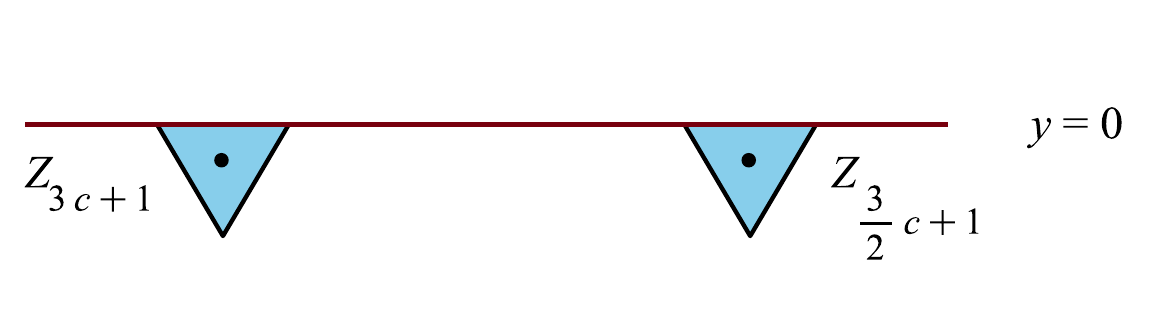}}
    \caption{Position of the tiles which intersect $y=0.$}\label{f:nonzerofreed1d2}
\end{figure}

We are going to call \emph{perfect triangles} the ones corresponding
to $Y_1,Y_2,\ldots ,Y_{3c+1}.$ As for perfect squares, we will prove
that these figures will evolve avoiding the discontinuity of $F.$
They are positioned as the Figure \ref{nonzerofree3} shows, 
the perfect triangles being the red ones, which are precisely the ones
pointing upwards. The blue ones correspond to $Z_1,Z_2,\ldots
,Z_{3c+1}.$
\begin{figure}[H]
    \centerline{\includegraphics[scale=0.4]{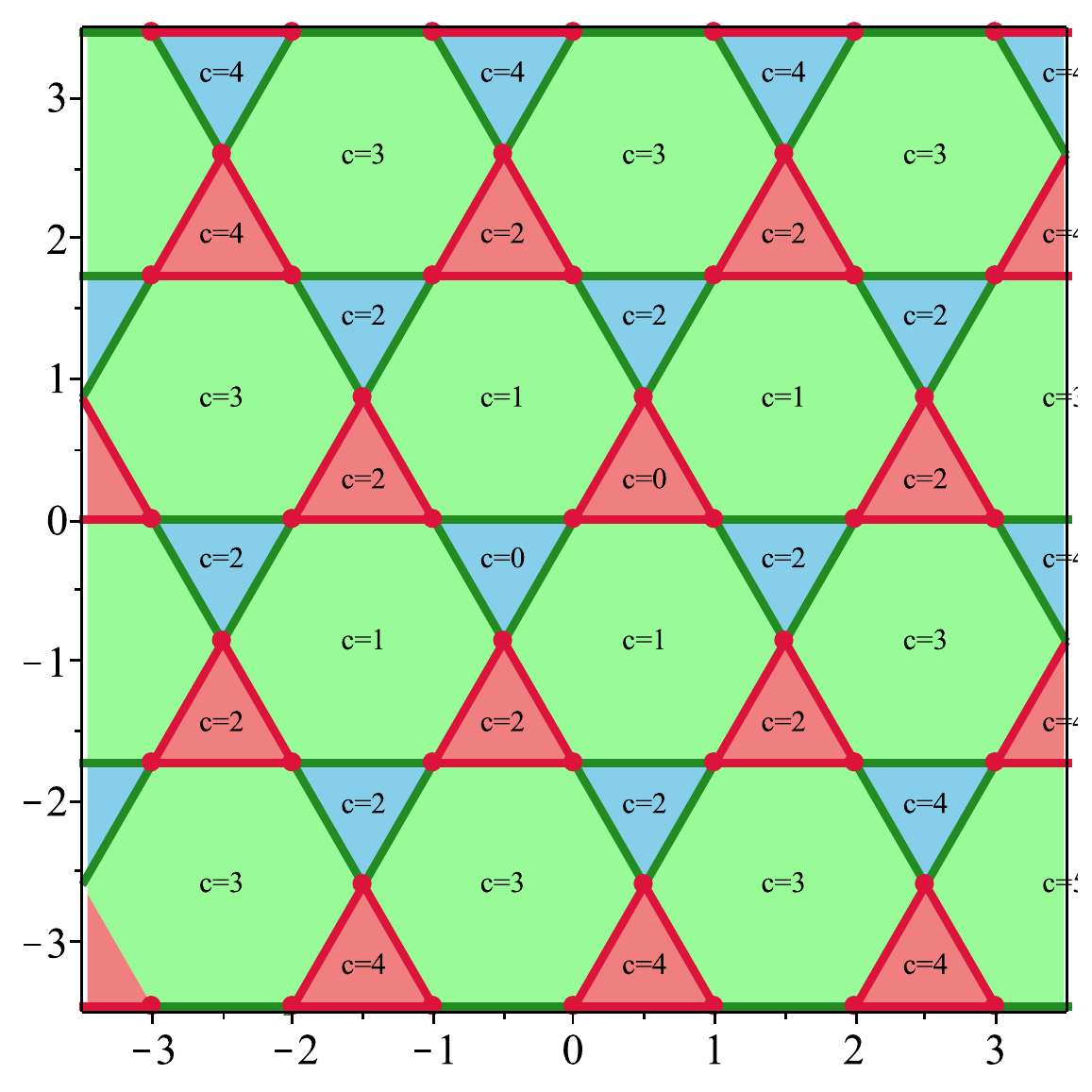}}
    \caption{Position of the perfect triangles in red. The borders and the vertices are highlighted with the
    color of the tile having the same itinerary map. The level set $\{V=c\}$ is indicated in each tile.}\label{nonzerofree3}
\end{figure}

\begin{proof}[Proof of item $(iii)$  of Theorem \ref{th:b}] First, we observe that the borders of the perfect triangles (including the vertices)
have the same period as the  interior points which are not centers.
Indeed, set an even number  $c\in \mathbb{N}_0$ and denote by $T_i$
the closed triangle (i.e. including the boundary with vertices)
which contain the point $Y_i$. For $i\ne 1,{3c}/{2}+1,$ the triangle
$T_i$ does not intersect $y=0,$ hence $F(T_i)=T_{i+c}.$ For $i=1$ or
$i={3c}/{2}+1,$ $F(T_i)=F_+(T_i)$ which also is $T_{i+c}.$
Therefore, by continuity, the points in the boundary of $T_i$ are
periodic with period $3(3c+1)=9c+3,$ as for the points in the interior
of $T_i.$

    Now take $c$ odd and let $H_i$ be the closed hexagon which contains $X_i$ in its interior.
    Looking at Figure \ref{nonzerofree3} we see that  $H_i$ has three edges which also are edges
    of a perfect triangle; if we call these edges the \emph{perfect edges}, we consider
    $\widetilde{H}_i =H_i\setminus \text{perfect edges}.$ (the motivation for this name is similar that the
    ones of perfect triangles, or squares, and will be apparent later).  Then, $\widetilde{H}_i $
    contains three alternate edges, say $l_1,l_2,l_3$, such
    that the slopes of the straight lines which contain them  are $\sqrt{3},-\sqrt{3}$ and $0$ respectively. Observe that $l_3$
is always at the bottom of the hexagon, hence $\widetilde{H}_i $ is
always fully contained in $H_+$ and $H_-$, and therefore
$F(\widetilde{H}_i )=F_+(\widetilde{H}_i )$ or $F(\widetilde{H}_i
)=F_-(\widetilde{H}_i ).$ In any case the three edges included in
$F(\widetilde{H}_{i})$ are three alternate edges with the edge of
slope $0$ in the bottom of the hexagon $F(\widetilde{H}_{i}).$ That
is, $F(\widetilde{H}_{i})=\widetilde{H}_{i+c}$ for all
$i=1,2,\ldots,3c+1.$ As in the previous case, by continuity, the
points in the boundary of $\widetilde{H}_i $ are periodic with
period $9c+3,$ as for the points in the interior of $H_i.$ Hence we have
proved that all the points in the edges of the hexagons are periodic
points.

    It remains to consider the edges of the triangles which are not perfect. But, as can be seen in Figure
    \ref{nonzerofree3}, all these edges are also the edges of the contiguous hexagons, which we
    have already proved that all of them are periodic. Observe also that all the vertices belong to perfect triangles.
\end{proof}

\subsection{Proof of item $(iv)$ of Theorem \ref{th:b}}

As in the case $\alpha=\pi/2,$ the proof follows by replacing the
value of $V$ by $2n$ or $2n+1,$ for $n\in\N_0,$ in  the results of
the previous items. We re-obtain the results of  \cite{ChaWanChe12}.

\section{Proof of Theorem \ref{th:c}}\label{s:alpha1pi3}

\subsection{Preliminaries}

As in the case studied in the previous section, for each tile
$T_{k,\ell,m},$ the values $k,\ell,m$ are not independent. Here,
either $m=k+\ell-1$ or $m=k+\ell$ or $m=k+\ell+1.$

\begin{lem}\label{l:centroid2}
    Let $T_{k,\ell,m}$ be one tile of  $\mathcal{U}_{\pi/3}=\mathcal{U}$.
        \begin{enumerate}[(a)]
                \item If $m=k+\ell$ then $T_{k,\ell,m}$
            is a hexagon whose center  is
            $p_{k,\ell}=\left(2k+\ell+{1}/{2},\sqrt{3}\ell+{\sqrt{3}}/{2}\right).$
            \item If $m=k+\ell-1$ or $m=k+\ell+1$, then $T_{k,\ell,m}$
            is a triangle whose center is either
            $q_{k,\ell}=\left(2k+\ell-{1}/{2},\sqrt{3}\ell+{\sqrt{3}}/{6}\right)$
            or
            $r_{k,\ell}=\left(2k+\ell+{3}/{2},\sqrt{3}\ell+{5\sqrt{3}}/{6}\right),$
            respectively.
\end{enumerate}
\end{lem}

\subsection{Proof of items $(i)$ and $(ii)$ of Theorem \ref{th:c}: dynamics on the zero-free set}

These results can be proved by the same arguments that we have used
in the proofs of Theorems \ref{th:a} and \ref{th:b}, in Sections
\ref{s:alphapi2} and \ref{s:alpha2pi3}.  Although we will not give
all the details of their proofs,
 we want to highlight the main features and results that allow to give the dynamics  in this case.

Consider  an even number $c$. Then, by Lemma \ref{l:centroid2}, the
tiles on the level set $V=c$ are hexagons whose  centers are some of
the points $p_{k,\ell}$  for some $k,\ell.$
 It can be proved that there are $3c+2$ centers in this level. This centers lie in certain hexagons. We
 denote them by $\{X_1,X_2,\ldots ,X_{3c+2}\}$ labeling them as in the case $\alpha={2\pi}/{3},$ see the Figure \ref{f:alpha60c=4v2}.
 \begin{figure}[H]
 \centerline{\includegraphics[scale=0.4]{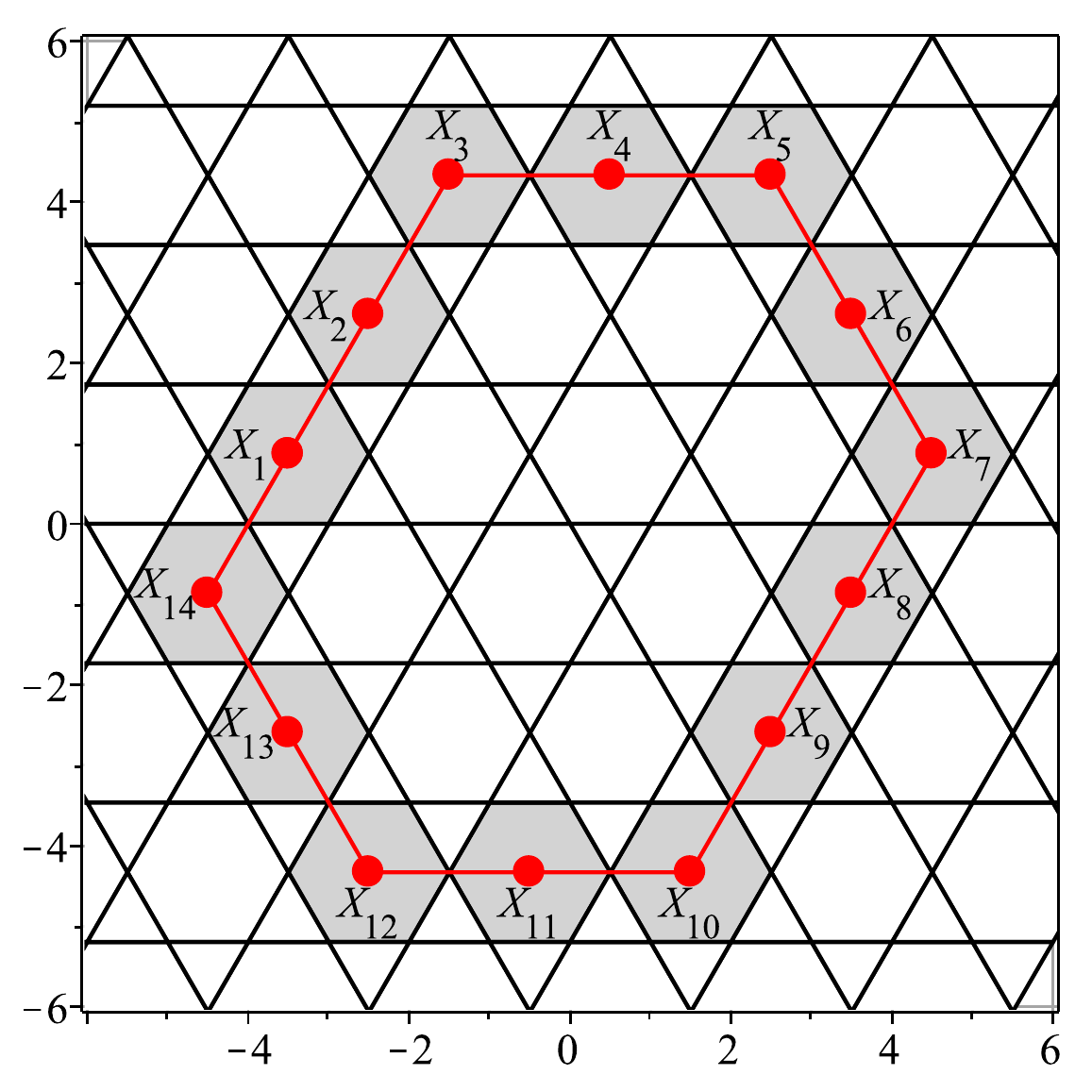}}
    \caption{Position of the centers in the level  $c=4.$ Observe that the points do not belong to the same orbit.
    In this case there are two different orbits. The lines linking the centers are plotted because their expressions play a role in
    order to obtain the expression of the first integrals $V$, as explained in Section \ref{s:firstintegrals}.}\label{f:alpha60c=4v2}
 \end{figure}

In this case $F$ restricted to $\{X_1,X_2,\ldots ,X_{3c+2}\}$ is conjugated to
\begin{equation}\label{e:h3cmes2}
h:\Z_{3c+2}\longrightarrow \Z_{3c+2}\mbox{ where } h(i)=i+{c}/{2}.
\end{equation}
From this equality, and using that $3c\equiv -2$ mod $3c+2$, one
easily gets that when ${c}/{2}$ is even then the minimal period is
$3{c}/{2}+1$,   and that when ${c}/{2}$ is odd, then the minimal
period is $3c+2$. In  the first case we get two periodic orbits
$\{X_1,X_3,\ldots ,X_{3c+1}\}$ and $\{X_2,X_4,\ldots ,X_{3c+2}\}$
while in the second one all the points $X_i$, $i=1,2,\ldots ,3c+1$
belong to the same periodic orbit.

To study the periodicity of the points in the hexagonal tile different from its center, for each $X_j$ we consider its itinerary map $I_j.$
 \begin{itemize}
    \item When ${c}/{2}$ is even, $I_j$ has the form $I_j=A^{3{c}/{2}+1}+v$ (where $v=X_j-A^{3{c}/{2}+1}X_j$)
    and $3{c}/{2}+1=3\cdot 2n+1=6n+1$ for some $n\in\N_0.$ Hence, since $A^6=\operatorname{Id},$ it holds that $I_j=A+v.$
     Therefore $I_j$ restricted to the hexagon which contains
    $X_j,$ is a rotation of angle ${\pi}/{3}$ centered at $X_j$ and every point in the hexagon is a 6-periodic point for $I_j.$
    It implies that these points are $6(3{c}/{2}+1)=9c+6$ periodic points for $F.$

    \item When ${c}/{2}$ is odd, $I_j=A^{3c+2}+v$ and $3c+2=3\cdot 2(2n+1)+2=6(2n+1)+2$ for some $n\in\N_0.$ Thus
    $I_j=A^2+v,$ using again that $A^6=\operatorname{Id}.$ This implies that every point in the hexagon is a $3(3c+2)=9c+6$ periodic point for $F.$
    \end{itemize}

     Now let $c$ be an odd number. Then the tiles in $\{V=c\}$ are triangles whose centers are either the points $q_{k,\ell}$ or
     the points $r_{k,\ell}$ introduced in Lemma \ref{l:centroid2}, for some $k,\ell\in\Z$. The points $q_{k,\ell}$ lie in some
     lines that define a hexagon, as do the points $r_{k,\ell}.$ But now all these centers belong to the same periodic orbit.
     To prove this, as usual, we label these points in a clockwise direction, as Figure \ref{f:alpha60c=3} shows for the case $c=3$: the
     red points are the points $q_{k,\ell}\in\{V=3\}$ while the blue ones are $r_{k,\ell}\in\{V=3\}.$
\begin{figure}[H]
    \centerline{\includegraphics[scale=0.4]{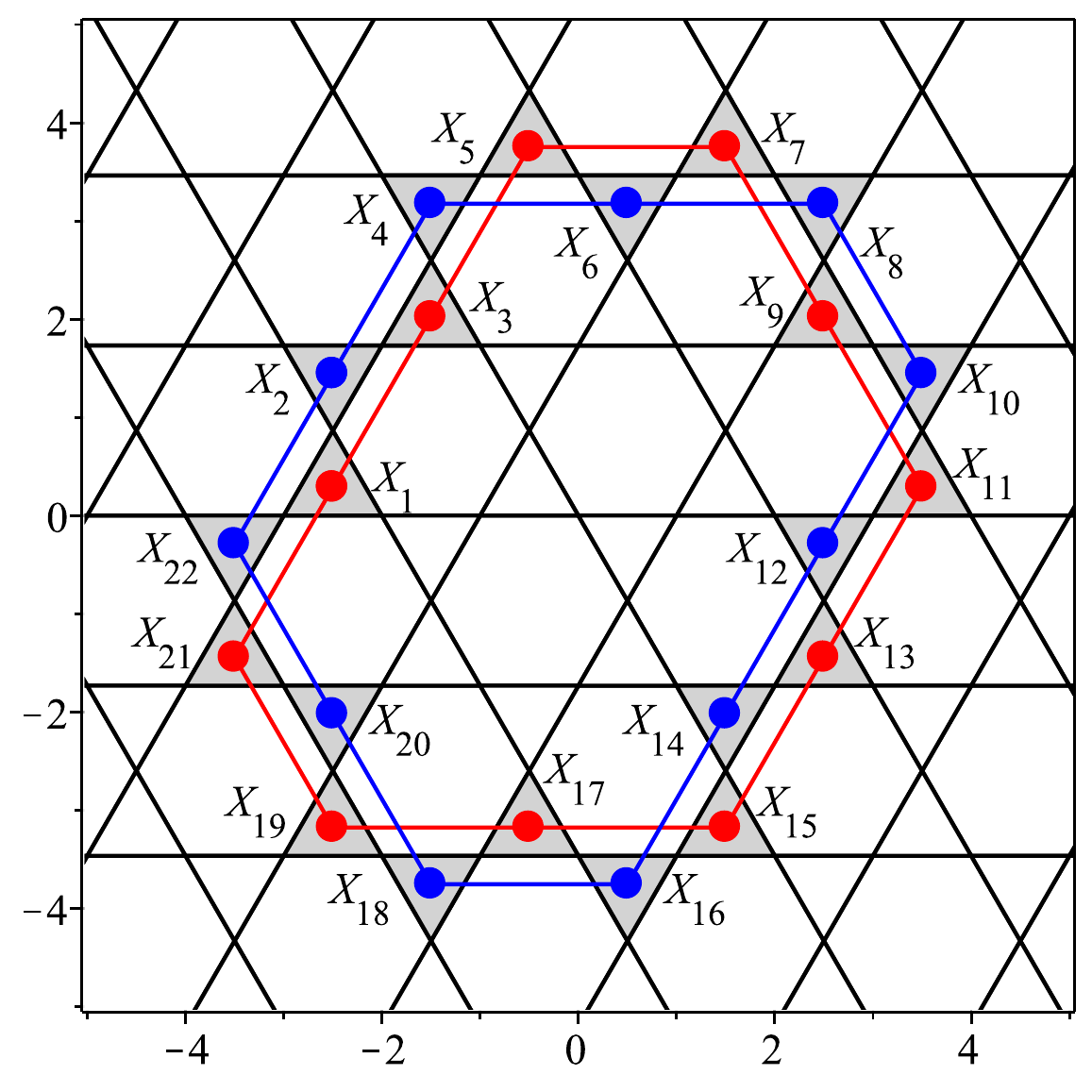}}
    \caption{Position of the centers in the level  $c=3.$ Observe that all the points belong to the same orbit. As shown in Section
    \ref{s:firstintegrals}, the lines joining the centers play a role in the determination of the first integral $V$ and do not
    represent two different orbits.}\label{f:alpha60c=3}
\end{figure}

With this labeling it can be proved that $F(X_i)=X_j$ with $j\equiv
i+c$ (mod $6c+4$) which implies that the minimal period is $p=6c+4.$
To see the periodicity of the points in the triangles different from
its center we consider the itinerary function  of the center $X_j$
which has the form $I_j=A^4+v.$ Then
$I_j^3=A^{12}+(A^4+A^2+\operatorname{Id})v=\operatorname{Id}.$
Arguing as before we get that each point in the triangle different
from its center is a $3(6c+4)=18c+12$ periodic point.

\subsection{Proof of item (iii) of Theorem \ref{th:c}: dynamics on the non zero-free set}

In this case, the dynamics of the points on the edges and vertices
of the tiles is more complicated than the ones found in the cases
$\alpha={\pi}/{2}$ and $\alpha={2\pi}/{3},$ so we are going to give
the details.

\subsubsection{Perfect edges and vertices}

We begin by considering the levels $c=4k,$ $k\in\N.$ We already know
that in these levels there are $3c+2=12k+2$ centers,
$X_1,X_2,\ldots,X_{12k+2}$ and $F(X_i)=X_j$ where $j=i+2k$ mod
$(12k+2).$ Also $X_1,X_3,\ldots ,X_{12k+1}$ form a periodic  orbit
of period $6k+1$, as does $X_2,X_4,\ldots X_{12k+2}.$ Let
$H_j$ be the hexagon such that $X_j\in H_j$ \emph{including its
boundary} (hence also its vertices). Then the hexagons that meet
$y=0$ are $H_1,H_{6k+1}$ (its bottom edge is contained in $y=0$) and
$H_{6k+2},H_{12k+2}$ (its top edge is contained in $y=0$). See
Figure \ref{f:fig0seccio5v3}.
\begin{figure}[H]
    \centerline{\includegraphics[scale=0.5]{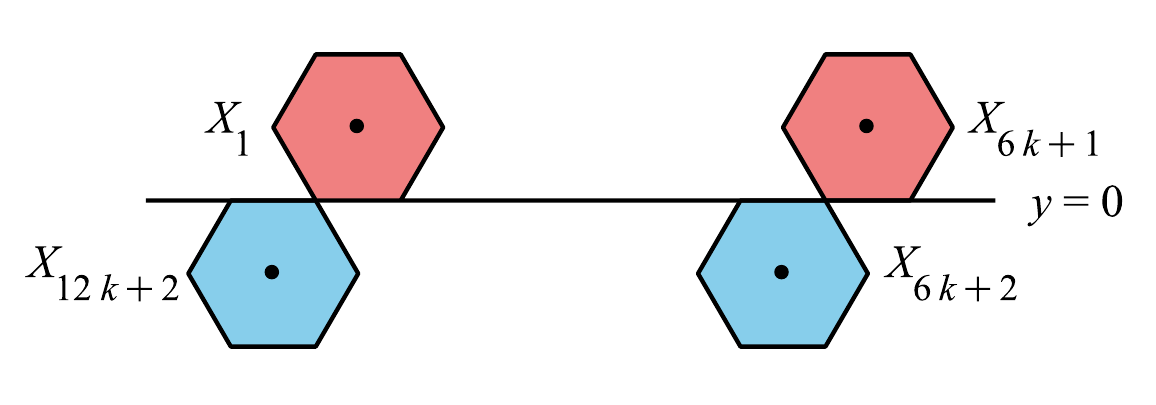}}
    \caption{Position of the hexagonal tiles of a level $c=4k$  which intersect $y=0.$ The red ones belong to the set of
    perfect hexagons.}\label{f:fig0seccio5v3}
    \end{figure}

Clearly for all $j=1,2,\ldots ,6k+1$, $F(H_j)=F_+(H_j)=H_{j+2k}$.
while for $j=6k+3,6k+4,\ldots, 12k+1$ also
$F(H_j)=F_{-}(H_j)=H_{j+2k\mbox{ mod }(12k+2)}.$ But the hexagons
$H_{6k+2},H_{12k+2}$ do not satisfy this property because on the top
edge of these hexagons $F=F^+$. Then we easily get:
\begin{lem} Assume that $c=4k$ and consider the (closed) hexagons $H_1,H_3,\ldots ,H_{12k+1}.$ Then for all $j=1,3,\ldots, 12k+1$
every point in $H_j$ different from its center is periodic of period
$36k+6.$
    \end{lem}
\begin{proof} For $j=1,3,\ldots, 12k+1$, the hexagons  satisfy $F(H_j)=H_{j+2k\mbox{ mod } (12k+2)},$ hence it is easy to
observe that their images are never the hexagons $H_{6k+2}$ and
$H_{12k+2}$. Then, by continuity, every point on the  boundary of
$H_j$ has the same periodicity as the points inside the hexagon (except the center). In
particular, $H_1,H_3,\ldots ,H_{12k+1}$ form an invariant set.
\end{proof}
 As in the above sections we
call $H_1,H_3,\ldots ,H_{12k+1}$ \emph{perfect hexagons} and their
edges  and vertices behave as the corresponding interior points, apart from the
 centers, that is they are $(36k+6)-$ periodic. Also we will call
\emph{non-perfect edges or vertices} those which do not collide with
a perfect hexagon.

\begin{figure}[H]
    \centerline{\includegraphics[scale=0.4]{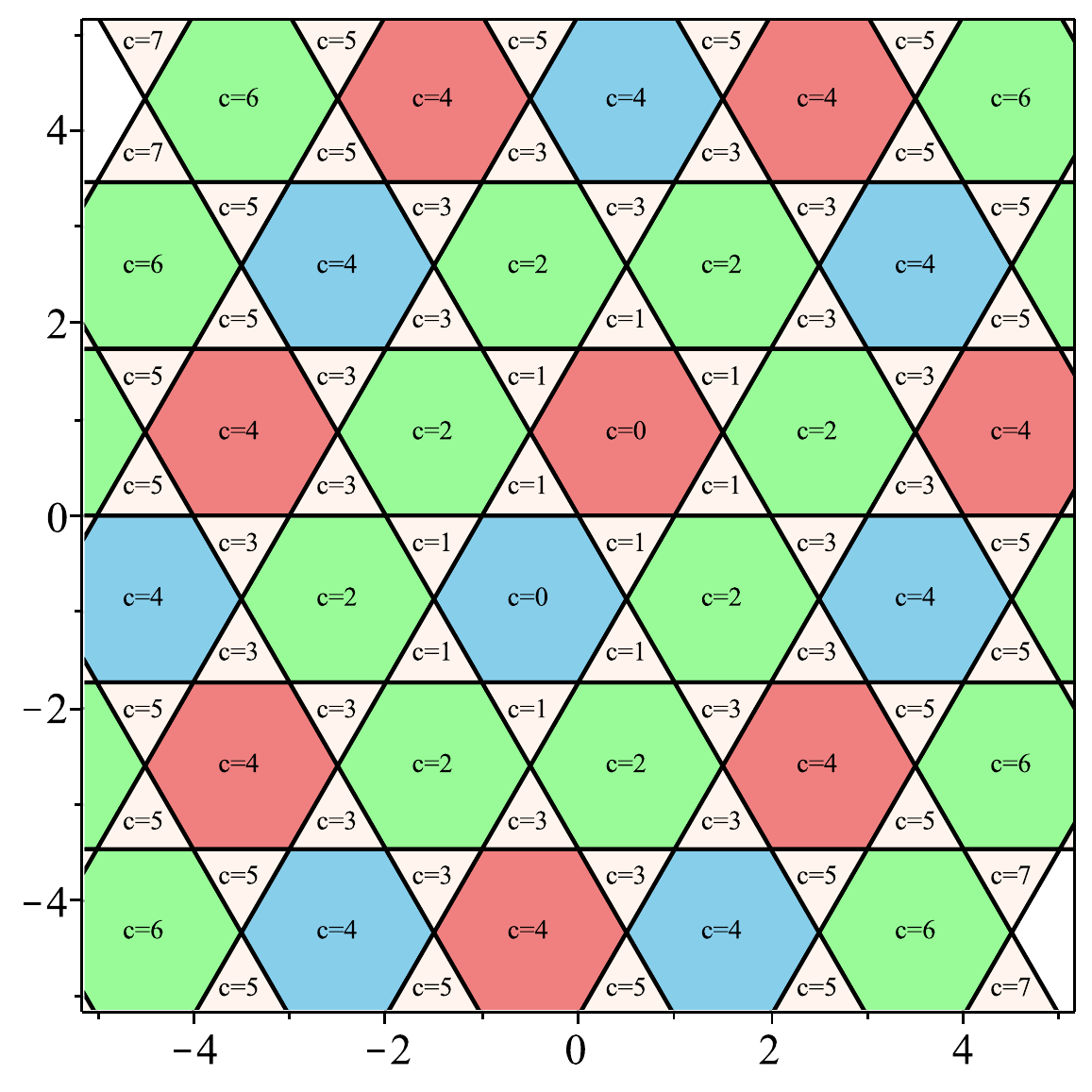}}
    \caption{Position of the perfect hexagons in red. The level  $\{V=c\}$ is indicated in each set.}\label{f:zerofree4}
\end{figure}

\subsubsection{Non-perfect edges}\label{sss:nonperf}

We continue the study considering the even levels of the form
$c=4k+2$. We know that in these level sets there are $3c+2=12k+8$
centers and that all of them belong to the same periodic orbit. The
hexagons which meet $y=0$ are $H_1,H_{6k+4},H_{6k+5}$ and
$H_{12k+8}$, see Figure \ref{f:fig1seccio5v2}.
\begin{figure}[H]
    \centerline{\includegraphics[scale=0.5]{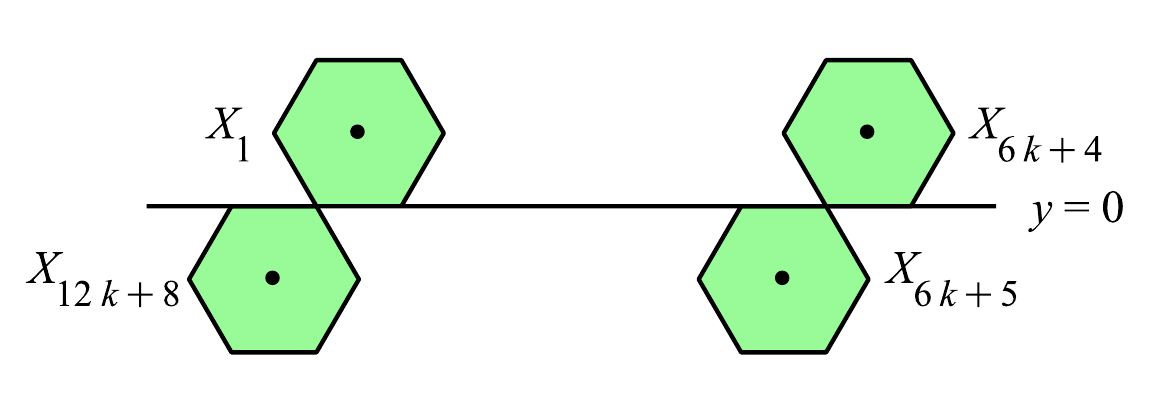}}
    \caption{Position of the hexagonal tiles of a level $c=4k+2$, which intersect $y=0.$}\label{f:fig1seccio5v2}
    \end{figure}

We are going to follow the dynamics of the interior  of the bottom
edge of $H_1,$ that we will denote as $\mathcal{L}$ (for simplicity
we will use the term \emph{edge} although the two boundary points
are not included). This dynamics is, by far, the most complex of
those we have studied in this paper. Since the argument is long, we
first briefly summarize it: \emph{we will show that every point in
$\mathcal{L}$ is $(108k+72)$-periodic. The edge is rigidly mapped by
by iterating $F$ into the edges of the hexagons in the level $4k+2$,
but also into the edges of the triangles in the levels $4k+1$ and
$4k+3$.}

Indeed, after the first iteration, $F(\mathcal{L})$ is the edge of
$H_{2k+2}$ obtained after rotating $\mathcal{L}$ an angle equal to
${\pi}/{3}$, because $F(X_1)=X_{2k+2}$ (remember that
from~\eqref{e:h3cmes2}, $F(X_i)=X_{i+2k+1\mbox{ mod } (12k+8)}.)$ We
continue iterating until we find the hexagon $H_{6k+5}.$ To compute
how many iterations we need for $X_1$ to reach $X_{6k+5}$ we ask for
the minimal positive number $p$ such that $F^p(X_1)=X_{6k+5}.$ That
is, $F^p(X_1)=X_{6k+5},$ or equivalently,  $1+p(2k+1)\equiv
6k+5\mbox{ mod }(12k+8).$ Thus,
\begin{align*}
p\equiv (2k+1)^{-1}(6k+4)= (6k+1)(6k+4)=36k^2+30k+4\equiv 6k+4\mbox{
mod }(12k+8).
\end{align*}

That is $F^{6k+4}(X_1)=X_{6k+5}.$ Now $F^{6k+4}(\mathcal{L})$ is the
edge of $H_{6k+5}$ after rotating $\mathcal{L}$ an angle equal to
$\frac{4\pi}{3}.$  Hence we follow iterating until we arrive to
$X_{12k+8},$ that is, three iterates more:
$F^3(X_{6k+5})=X_{12k+8}.$ This implies that $F^{6k+7}(\mathcal{L})$
is the edge of $H_{12k+8}$ obtained after rotating $\mathcal{L}$ an
angle equal to ${\pi}/{3}.$ Now we ask for the minimal $p$ such that
$F^p(X_{12k+8})=X_{6k+5}.$ The computation gives that $p=12k+5.$
Hence we can write
$$X_1\overset{F^{6k+4}}{\longrightarrow}X_{6k+5}\overset{F^{3}}{\longrightarrow}X_{12k+8}\overset{F^{12k+5}}{\longrightarrow}
X_{6k+5}\overset{F^{3}}{\longrightarrow}X_{12k+8},$$ and following
the edge $\mathcal{L}$ we have that after $18k+15$ iterates the
initial edge $\mathcal{L}$ of $H_1$ is  the top edge of $H_{12k+8},$
\emph{which is the bottom edge of the triangle $T_1$ in the level
$4k+3$.}

Next, we follow the orbit of the centers of the triangles in the
level $c=4k+3$. Let $Y_1,Y_2,\ldots, Y_{24k+22}$ be the centers of
the triangles $T_1,T_2,\ldots ,T_{24k+22}$. All of them form a
unique periodic orbit  and $F(Y_i)=Y_j$ where $j=i+4k+3$ mod
$(24k+22)$, remember that in the triangles $F(Y_i)=Y_{i+c\mbox{ mod
} (6c+4)}$. The triangles with edges in the critical line are
displayed in the Figure \ref{f:fig2seccio5v2}.
\begin{figure}[H]
    \centerline{\includegraphics[scale=0.5]{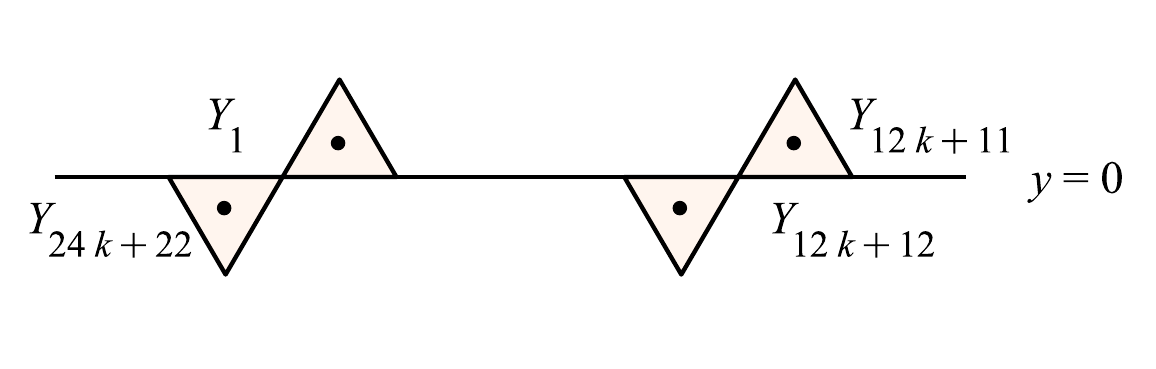}}
        \caption{Position of the triangular tiles which intersect $y=0.$}\label{f:fig2seccio5v2}
\end{figure}

 Taking into account that $(4k+3)^{-1}=18k+15$ in $\Z_{24k+22}$ and solving the corresponding congruences we find that:
$$Y_1\overset{F^{6k+7}}{\longrightarrow}Y_{24k+22}\overset{F^{6k+4}}{\longrightarrow}Y_{12k+12}\overset{F^{18k+18}}
{\longrightarrow}Y_{24k+22}\overset{F^{6k+4}}{\longrightarrow}X_{12k+12}.$$
Then we see that the bottom edge of $T_1$ is transformed into the
top edge of $T_{12k+12}$ after $36k+33$ iterates.  This one also is
the bottom edge of the hexagon $H_{6k+4}$ in the level $c=4k+2,$ see
again Figure \ref{f:zerofree4}, and also Figure \ref{f:T60}.

Following the same procedure it can be seen that
$F^{6k+1}(X_{6k+4})=X_{6k+5}$ and using the calculations made before
we obtain
$$X_{6k+4}\overset{F^{6k+1}}{\longrightarrow}X_{6k+5}\overset{F^{3}}{\longrightarrow}X_{12k+8}\overset{F^{12k+5}}{\longrightarrow}X_{6k+5}.$$
Hence, the bottom edge of $H_{6k+4}$ is transformed into the top
edge of $H_{6k+5}$ after $18k+9$ iterates.

The top edge of $H_{6k+5}$ is also the bottom edge of one triangle
whose center belongs to the level set $4k+1.$ In this  level set
there are $24k+10$ centers of triangles, that we denote by
$Z_1,Z_2,\ldots ,Z_{24k+10},$ and we know that $F(Z_i)=Z_j$ with
$j=i+4k+1$ mod $(24k+10).$ We call $T_1,T_2,\ldots,T_{24k+10}$ these
triangles. Specifically, the top edge of $H_{6k+5}$ is the bottom
edge of $T_{12k+5}$, see Figures \ref{f:zerofree4} and
\ref{f:fig3seccio5v2}.
\begin{figure}[H]
    \centerline{\includegraphics[scale=0.5]{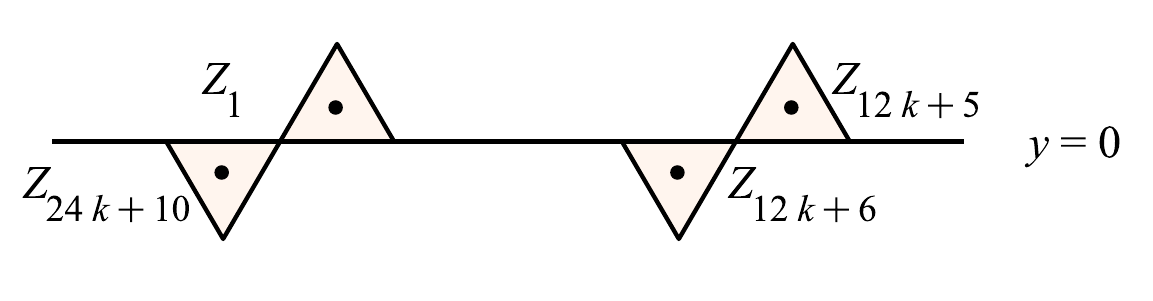}}
            \caption{Position of some of the triangular the tiles which intersect $y=0.$}\label{f:fig3seccio5v2}
\end{figure}

Using that in $\Z_{24k+10}$, $(4k+1)^{-1}=6k+1$ and solving the corresponding congruences we have that
$$Z_{12k+5}\overset{F^{6k+1}}{\longrightarrow}Z_{12k+6}\overset{F^{6k+4}}{\longrightarrow}Z_{24k+10}\overset{F^{18k+6}}
{\longrightarrow}Z_{12k+6}\overset{F^{6k+4}}{\longrightarrow}Z_{24k+10}.$$
It follows that after $36k+15$ iterates, the bottom edge of $T_{12k+5}$ is transformed in the top edge of $T_{24k+10}.$

But this top edge of $T_{24k+10}$ is exactly the edge $\mathcal{L}$.
Hence summing up the involved iterates we have that every  point in
$\mathcal{L}$ is a  $108k+72$ periodic point. Also the same holds
for all the points belonging to the $108k+72$ edges obtained
iterating $\mathcal{L}.$ In other words, we get a periodic orbit of
edges of period $108k+72$ and, of course, the points of
$\mathcal{L}$ are mapped to themselves after these iterations.

\subsubsection{Non-perfect vertices}

And what about the vertices? As we will see in the proof of Theorem
\ref{t:nonzerofreepi3}, we only need to prove the periodicity  of
the vertices in $y=0$. Observe that if such a vertex belongs to a
perfect hexagon, then we already know that it is periodic with the
same period as the interior points.  If it is non-perfect, then
either $(a)$ it is mapped to a vertex colliding from the top with a
triangle of level $V=4k+1$ and a hexagon of level $V=4k+2$, both in
$H_+$, as the solid-circle point in Figure \ref{f:vertex}; or (b) it
collides from the top with a hexagon of level $V=4k+2$  and triangle
of level $V=4k+3$, both in $H_+$, as the box-shaped point in
Figure~\ref{f:vertex}.

\begin{figure}[H]
    \centerline{\includegraphics[scale=0.7]{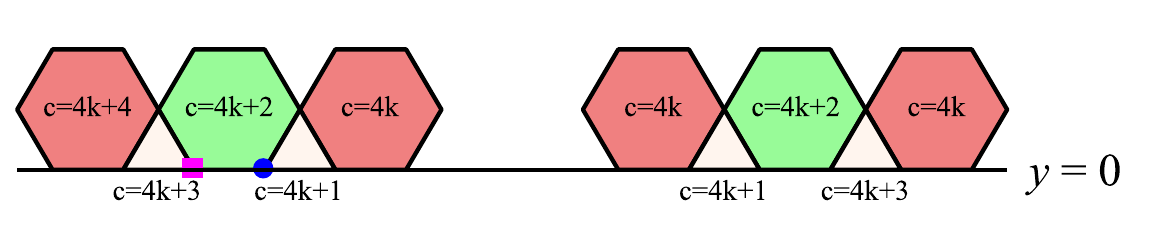}}
            \caption{Position of the non-perfect vertices in $y=0$. If $H_1$ is the left non-perfect hexagon with level
            $V=4k+2$ then the vertex $V_1(H_1)$  is the box-shaped point and the vertex $V_6(H_1)$ is the solid-circle point.
            The perfect hexagons are the red ones.}\label{f:vertex}
\end{figure}

To study the dynamics of the non-perfect vertices in $y=0$, we will
use the following notation: given a hexagon $H$, we label its
vertices as $V_i(H)$ with $i=1,\cdots,6$ starting from the
left-bottom vertex and in clockwise sense, see Figure
\ref{f:verthex}.

\begin{figure}[H]
    \centerline{\includegraphics[scale=0.35]{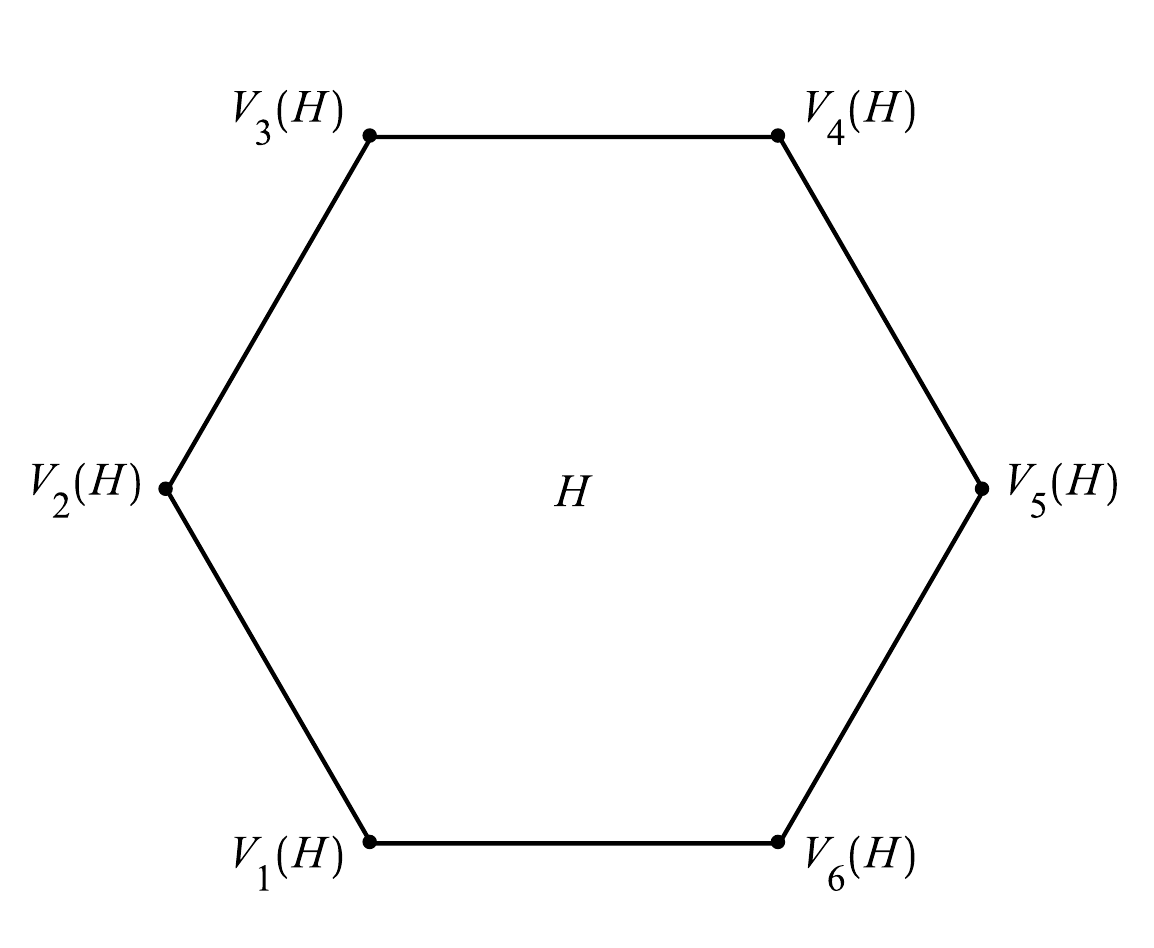}}
            \caption{Labeling the vertices of a hexagon $H$}\label{f:verthex}
\end{figure}

Let $H_1$ be the non-perfect hexagon at level $V=4k+2$ in $\mathcal{Q}_1$, whose intersection with $y=0$ is its bottom edge.
Then:

$(a)$ We will follow the orbit of the point $V_6(H_1)$ (the blue
point in Figure \ref{f:vertex}) by using the results found in
Section \ref{sss:nonperf}. In particular, we know  that $F^3(H_1)=H_{6k+4}$ ,
$F^{6k+1}(H_{6k+4})=H_{6k+5}$ , $F^3(H_{6k+5})=H_{12k+8}$ and  $F^{6k+1}(H_{12k+8})=H_{1}.$
Therefore, we easily find
$$V_6(H_1)\overset{F^{3}}{\longrightarrow}V_3(H_{6k+4})\overset{F^{6k+1}}{\longrightarrow}V_4(H_{6k+5})=V_1(H_{6k+4})
\overset{F^{6k+1}}{\longrightarrow}$$$$V_2(H_{6k+5})\overset{F^{3}}{\longrightarrow}V_5(H_{12k+8})
\overset{F^{6k+1}}{\longrightarrow}V_6(H_1),$$
hence the point $V_6(H_1)$ is $(18k+9)$-periodic.

(b) We will pursue the orbit of the point $V_1(H_1)$ (the box-shaped point in Figure \ref{f:vertex}).
$$V_1(H_1)\overset{F^{3}}{\longrightarrow}V_4(H_{6k+4})\overset{F^{6k+1}}{\longrightarrow}V_5(H_{6k+5})
\overset{F^{3}}{\longrightarrow}V_2(H_{12k+8})\overset{F^{6k+1}}{\longrightarrow}V_3(H_{1})
\overset{F^{3}}{\longrightarrow}$$$$V_6(H_{6k+4})\overset{F^{6k+1}}{\longrightarrow}V_1(H_{6k+5})
\overset{F^{3}}{\longrightarrow}V_4(H_{12k+8})=V_1(H_1),$$
hence the point $V_1(H_1)$ is $(18k+15)$-periodic.

Now we have all the ingredients to prove the main result of this
section, that clearly implies item $(iii)$ of Theorem \ref{th:c}.

\begin{teo}\label{t:nonzerofreepi3}
    Every non zero-free point of $F$ is periodic. Furthermore:
    \begin{enumerate}[(a)]
    \item If $(x,y)$ is a point in the edge of a perfect hexagon (which has energy level $V=4k$), then it is periodic with period $36k+6$.
    \item If $(x,y)$ is a point in a non-perfect (open) edge of a tile then it is periodic
    with period $108k+72$ for some $k\in\mathbb{N}_0$.
    \item If $(x,y)$ is a non-perfect vertex then it is periodic
    with period $18k+9$ or $18k+15$ for some $k\in\mathbb{N}_0$.
    \end{enumerate}
Observe that any non zero-free point belongs to one of the above cases.
\end{teo}

\begin{proof} We already know that the set of the non zero-free points is formed by the edges and the
vertices of the hexagons and triangles introduced before.

    Consider the points in one edge. Then after a finite number of iterates this edge is transformed in one edge contained in $y=0.$
If this edge correspond to an edge of a perfect hexagon with center
belonging to the level $4k,$ every  point will be
$6(6k+1)-$periodic. If not, it will be an edge of a polygon with
center belonging to the level either $4k+1,4k+2$ or $4k+3.$ From
the discussion above we know that every point will be
$(108k+72)-$periodic.

With respect to the vertices, observe that since any vertex belongs to
$\mathcal{F}$, after some iterates it will be  mapped to a vertex
point in $y=0$. Hence there are three possibilities: it is mapped to
a perfect vertex of a perfect hexagon in  $H_+$, which has energy level
$V=4k$ (and in this case it is periodic of period $36k+6$); or it is
mapped to a vertex colliding from the top with a triangle of level
$V=4k+1$ and a hexagon of level $V=4k+2$ (and in this case it is
periodic with period $18k+9$); or it is mapped to a vertex colliding
from the top with a hexagon of level $V=4k+2$  and a triangle of
level $V=4k+3.$ In this case it is periodic with period $18k+15.$
\end{proof}

The proof of item $(iv)$ is a straightforward consequence of all the
previous results.

\section{Obtaining the first integrals}\label{s:firstintegrals}

We have intuited the expressions of the first integrals after several simulations. For
completeness we present in detail a three-step constructive
procedure that allows to obtain the first  integrals of $F$ given in
\eqref{e:Vpi2} corresponding to $\alpha=\pi/2.$ For the other two
cases the line of argument is the same, and the details are analogous,
and we only give some comments.

\smallskip

\noindent{\bf Step 1.} By displaying some preimages of the critical
line $LC_{-i}$ we realize that the zero-free set is formed by open
tiles of a regular or uniform tessellation of $\R^2$. This fact is trivial in
this case where $\alpha=\pi/2$ but, a priori, it was not so obvious
in the cases $\alpha=2\pi/3$ and $\alpha=\pi/3$ studied in Sections
\ref{s:alpha2pi3} and \ref{s:alpha1pi3}. The normal form of $F$ given
in \eqref{e:normalform} regularizes the tesselation.

\smallskip

\noindent{\bf Step 2.} From preliminary numerical explorations we
also realize that the centers of some tiles form an invariant set
under the dynamics of $F.$ In the case $\alpha=\pi/2$ these centers
were located in the lines $y=x+c$, $y=-x+c+1$, $y=x-c$ and
$y=-x-c-1$ for a certain fixed value $c\in\mathbb{N}_0$, depending on
the quadrant where the center points are located (see Lemma
\ref{l:Nivell} and Figure \ref{f:figura1}, and see Lemmas
\ref{l:hexagons} and \ref{l:triangles} for the case $\alpha=2\pi/3$).

\smallskip

\noindent{\bf Step 3.} Isolating the value in the expression of the
lines linking the centers we obtain that $c=y-x$  for $
(x,y)\in\mathcal{Q}_1$, $c=y+x-1$ for $(x,y)\in\mathcal{Q}_2$,
$c=x-y$ for $(x,y)\in\mathcal{Q}_3$ and $c=-x-y-1$ for
$(x,y)\in\mathcal{Q}_4,$ where recall that $\mathcal{Q}_j,
j=1,2,3,4,$ are the four quadrants of $\R^2.$ From these expressions
and taking into account that $c\in\mathbb{N}_0$ and that given a
zero-free point $(x,y)$ the center point of its associated tile is
$\left(\operatorname{E}(x)+{1}/{2},\operatorname{E}(y)+{1}/{2}\right)$,
we arrive to the expression of the first integral $
V_{\pi/2}(x,y)=\max \left( \left|
\operatorname{E}(x)+\operatorname{E}(y)+1 \right| -1, \left|
\operatorname{E}(x)- \operatorname{E}(y) \right| \right). $

\section{Final comments.}

We have proved that for $\alpha\in\{\pi/3,\pi/2,2\pi/3\}$,
the corresponding zero-free sets $\mathcal{U}$ are the union of a countable number of open sets (the tiles), hence the associated critical sets $\mathcal{F}=\mathbb{R}^2\setminus\mathcal{U}$  are closed sets. In consequence, for any point $(x,y)\in\mathbb{R}^2$, the distance $\mathrm{dist}\left((x,y),\mathcal{F}\right)$ is well defined. Since $\mathcal{F}$ is also invariant, we have:
\begin{nota}\label{r:dist}
Any map \eqref{e:normalform} with $\alpha\in\{\pi/3,\pi/2,2\pi/3\}$ has the non-quantized continuous first integral $W(x,y)=\mathrm{dist}\left((x,y),\mathcal{F}\right)$.
\end{nota}

\medskip

We believe that the only pointwise periodic cases for the maps $F$ with $\alpha\in(0,2\pi)$, 
are the ones studied in this work as well the cases $\alpha\in\{4\pi/3,3\pi/2, 5\pi/3\}$ (recall that we where motivated by the study of the maps $G$ in \eqref{E:Grho}
with $|\rho|<2$, which are conjugated with the maps $F$ in \eqref{e:normalform}
with $\alpha\in(0,\pi)$).
In these later cases we have observed that the quantized first integrals given in this paper are also integrals of these maps. In particular: $V_{\pi/2}$, $V_{2\pi/3}$  and $V_{\pi/3}$ are first integrals of $F$ when  $\alpha=3\pi/2, 5\pi/3$ and $4\pi/3$ respectively. However notice that 
none of these last maps are conjugated to the maps considered in this work: for example, note that the maps $F$ with  $\alpha\in\{4\pi/3,3\pi/2, 5\pi/3\}$ do not have fixed points since the centers of the rotations are virtual.

The maps $F$ belong to the class of symmetric maps studied in the relevant paper \cite{GQ}. We refer the reader to this reference to learn about the general properties of the maps $F$ with
$\alpha$ a general value in $[0,2\pi)\setminus\{{\pi}/{3},
\pi/{2},2\pi/{3},4\pi/3,3\pi/2, 5\pi/3\}$. For instance, in that paper it is proved that for any $\alpha\ne\pm\pi$ being a rational multiple of $\pi$ there exists a sequence of open invariant nested necklaces,  that tend to infinity,  each one of them being similar  to the level sets of our quantized first integrals, whose beads are polygons,  and where the dynamics of $F$ is given by a product of two rotations.  Remarkably, although the adherence of the union of all these invariant necklaces does not fill the full plane, it  allows to prove that all orbits of $F$ are bounded.

\medskip

\centerline{\textbf{Acknowledgements}} We want to thank the anonymous reviewers for their comments that have allowed us to better contextualize the problem, improve our article and have provided us with very interesting references. The authors are supported by
Ministry of Science and Innovation--State Research Agency of the
Spanish Government through grants PID2019-104658GB-I00  (first and
second authors),  DPI2016-77407-P
 (AEI/FEDER, UE, third author) and MTM2017-86795-C3-1-P (fourth autor). The first, second and fourth authors are also supported by the grant 2017-SGR-1617  from
AGAUR,  Generalitat de Catalunya. The third author acknowledges the
group's research recognition 2017-SGR-388 from AGAUR, Generalitat de
Catalunya.

\end{document}